\documentclass[mn,fleqn,a4paper]{w-art}
\usepackage{times}
\usepackage{w-thm}
\usepackage[utf8]{inputenc}
\usepackage{latexsym}
\usepackage{amsthm}
\usepackage{amssymb}
\usepackage{amsmath}
\usepackage{amsfonts}
\usepackage{indentfirst}
\newcommand{\OK}{{\overline{{\mathbb K}}}}

\newcommand{\bc}{\begin{center}}
\newcommand{\ec}{\end{center}}

\newcommand{\tmsamp}[1]{\textit{#1}}
\theoremstyle{definition}
\newtheorem{Def}[theorem]{Definition}
\newtheorem{The}[theorem]{Theorem}
\newtheorem{Pro}[theorem]{Proposition}
\newtheorem{Lem}[theorem]{Lemma}
\newtheorem{Cor}[theorem]{Corollary}
\newtheorem{Rem}[theorem]{Remark}

\newtheorem{Not}[theorem]{Notation}
\DeclareMathOperator{\diam}{diam}
\DeclareMathOperator{\supp}{supp}
\DeclareMathOperator{\Inv}{Inv}
\DeclareMathOperator{\dis}{d}
\DeclareMathOperator{\Ann}{Ann}
\DeclareMathOperator{\Rad}{Rad}
\DeclareMathOperator{\Arg}{Arg}
\DeclareMathOperator{\dd}{d}
\DeclareMathOperator{\A}{A}
\DeclareMathOperator{\V}{V}
\DeclareMathOperator{\cl}{cl}

\def\Rset{{\mathbb R}}
\def\Zset{{\mathbb Z}}
\def\Cset{{\mathbb C}}
\def\Nset{{\mathbb N}}
\def\Kset{{\mathbb K}}
\def\Qset{{\mathbb Q}}

\begin{document}

\title[Colombeau's Algebra of full Generalized Numbers]{Colombeau's Algebra of full Generalized Numbers
}
%%
%%    Information for the fist author
\author[Sh. J.~Aragona]{Jorge Aragona\footnote{E-mail: {\sf aragona@ime.usp.br},}\inst{1}}
 \address[\inst{1}]{Universidade de São Paulo, Instituto de
   Matemática e Estatística}
%%
%%    Information for the second author
\author[Sh. A.~R.~G.~Garcia]{Antonio Ronaldo Gomes
  Garcia\footnote{E-mail: {\sf ronaldogarcia@ufersa.edu.br}}\inst{2}}
\address[\inst{2}]{Universidade do Estado do Rio Grande do Norte,
  Departamento de Matemática e Estatística\\
Universidade Federal Rural do Semi-Árido, Departamento de Ciências Ambientais}
%%
%%    Information for the third author
\author[Sh. S.~O.~Juriaans]{Stanley Orlando Juriaans
\footnote{E-mail: {\sf ostanley@ime.usp.br},}\inst{1}}

\maketitle

\begin{abstract} 
Let $\overline{\Kset}_f$ denote the commutative unital ring of  
Colombeau's full generalized numbers. This ring can be endowed 
with an ultra-metric in such a way that it becomes a topological 
ring. There are many interesting question about $\overline{\Kset}_f$ 
in the framework of Commutative Algebra and General Topology as well 
as of the superposition of these two subjects. The purpose of this 
paper aims to give an initial step toward the study of this ring.  
\end{abstract}

\section*{Introduction} In what follows $\Kset$ will denote either
$\Rset$ or $\Cset$ and $\overline{\Kset}_f$ will denote the
commutative ring of Colombeau's full generalized
numbers. Recently, Aragona and Juriaans (see \cite{AJ})
developed algebraic and topological methods to study the ring 
of the simplified algebra of Colombeau generalized numbers 
$\overline{\Kset}$ (see \cite{AJ}). Understanding the algebraic 
and topological properties of $\OK$ is important because a 
generalized function can be considered to be an $C^{\infty}$ 
function defined on a subset of $\OK^n$ as shown in \cite{afj}.  

The purpose of this paper is to give a contribution to the study of 
the algebraic and topological properties of $\overline{\Kset}_f$. 
In Section \ref{sec-01} we collect the basic
definitions, results and notation that will be used in the throughout
the paper and as a rule most of the proofs are omitted. In Section
\ref{sec-02} we present our first results most of them being 
extensions of results of Aragona-Juriaans. 
In Section \ref{sec-03}, we present a more accurate look at the 
structure of the maximal ideals of $\overline{\Kset}_f$ 
based on two basic tools. First, a careful analysis of the set of the
representatives of elements of $\overline{\Kset}_f$. Second, the
introduction of a set $\mathcal{S}_f$ of subsets of the domain
$\mathcal{A}_0(\Kset)$ of the representatives of elements of 
$\overline{\Kset}_f$. The set $\mathcal{S}_f$ plays the "role" of the
set $\mathcal{S}$ in the article of Aragona and Juriaans 
(see Definition 4.1 in \cite{AJ}). The family $\{\mathcal{X}_A\}
~(A\in\mathcal{S}_f)$ of elements of $\overline{\Kset}_f$, where
$\mathcal{X}_A$ is the characteristic function of
$A$, is very useful to the study of a number of interesting 
properties one of them being that the unit group $\overline{\Kset}_f$ is 
dense. We also derive several characterizations of the units of
$\overline{\Kset}_f$ as well as a complete description of the maximal
ideals of $\overline{\Kset}_f$. The description of the
prime ideals of $\overline{\Kset}_f$ appears to be more complicated
than that of the maximal ones. Nevertheless we obtain a first step in this
direction. Finally, in Section, \ref{sec-04} 
based in the work of Aragona, Juriaans, Scarpalezos and Oliveira
\cite{AJOS} we study  order relations on
$\overline{\Rset}_f$  which is used  to continue with the study of 
the algebraic properties of $\overline{\Kset}_f$. Indeed, 
we describe completely the minimal primes and show that
$\overline{\Kset}_f$ is not Von Neumann regular.
\section{The sharp topology on $\overline{\Kset}_f$}\label{sec-01}
In this section we recall some basic definitions and results about
$\overline{\Kset}_f$ with the purpose of fixing the terminology. As a
rule, the proofs are be omitted. 

\newpage
\begin{Not}{\tmsamp{Some notations necessary in this work
\begin{enumerate}
\item[$(a)$] $I:=]0,1], ~\overline{I}:=[0,1]$ e $I_\eta:=]0,\eta[,
~\forall~\eta\in I$.
\item[$(b)$] $A\setminus B:=\{a\in A|a\notin B\}$. 
\item[$(c)$] $\Qset$, denotes the field rational numbers.
\item[$(d)$] $\Kset$, denotes the field of real (or complex) numbers,
  i.e., $\Rset$ (or $\Cset$).
\item[$(e)$] $\Kset^*:=\Kset\setminus \{0\}$.
\item[$(f)$] $\Nset, \Zset$, stands respectively for 
the set of the natural numbers and the set of integers.\\ 
$\Nset^*:=\Nset\setminus\{0\}$ and $\Zset^*:=\Zset\setminus\{0\}$.
\item[$(g)$] $\Kset^*:=\Kset\setminus \{0\}$.
\item[$(h)$] $\Rset_+:=\{x\in\Rset|x\ge 0\}$ e $\Rset^*_+:=\{x\in\Rset|x>0\}$.
\item[$(i)$] $\overline{\Kset}$, denotes the ring of Colombeau's
  simplified generalized numbers. 
\item[$(j)$]
  $\mathcal{A}_0(\Rset):=\left\{\varphi\in\mathcal{D}(\Rset)|\int_0^\infty
  \varphi(x)\dd x=\frac{1}{2}, ~\varphi ~\mbox{is even 
    and}~\varphi\equiv\mbox{const.}~\mbox{em}~ V_0\right\}$, where $V_0$
  is a neighborhood of the origin.
\item[$(l)$] $\mathcal{A}_q(\Rset):=\left\{\varphi\in\mathcal{A}_0(\Rset)|\int_0^\infty x^
{\frac{j}{m}}\varphi(x)\dd x=0, ~\mbox{para}~1\le j,m\le q\right\}$, where 
$q\in \Nset$.
\item[$(m)$] $\Gamma:=\{\gamma:\Nset\to \Rset_+|\gamma(n)<\gamma(n+1),
  ~\forall~n\in \Nset ~\mbox{and}~\lim_{n\to
    \infty}\gamma(n)=\infty\}$ is the set of the strict increasing
  sequences diverging to infinity when $n\to\infty$.
\end{enumerate}}}
\end{Not}

\begin{Def}\label{simplificada}{\tmsamp{
Let $\mathcal{E}_f(\Kset)$ be the ring (pointwise operations) of
  the functions $v:\mathcal{A}_0(\Kset)\to\Kset$. Let
  $\mathcal{E}_f^M(\Kset)$ denote the subring of
  $\mathcal{E}_f(\Kset)$ of those functions $v$ satisfying:}}
{\tmsamp{
\begin{enumerate}
  \item[$(M)$]\label{kset-1} $\exists ~p\in\Nset$ such that 
$\forall ~\varphi\in\mathcal{A}_p(\Kset), 
~\exists~ C=C_\varphi>0, ~\exists~\eta=\eta_\varphi>0$ such that
$$|v(\varphi_{\varepsilon})|\le C\varepsilon^{-p}, ~\forall~
0<\varepsilon<\eta.$$
\end{enumerate}
These are called moderate functions.}}

{\tmsamp{Let $\Gamma$ be as in $(m)$ and define $\mathcal{N}_f(\Kset)$ the
ideal of $\mathcal{E}_f^M(\Kset)$ of those functions $v$ satisfying:
\begin{enumerate}
\item[$(N)$]\label{kset} $\exists~ p\in\Nset,
  ~\exists~\gamma\in\Gamma$ such that
$\forall~q\ge p ~\mbox{e}~\forall~\varphi\in\mathcal{A}_q(\Kset),
~\exists~ C=C_\varphi>0, ~\exists~\eta=\eta_\varphi>0$ such that
$$|v(\varphi_\varepsilon)| \le C\varepsilon^{\gamma(q)-p},
~\forall~0<\varepsilon<\eta.$$
\end{enumerate}
These are called null functions.}}

{\tmsamp{The ring of the Colombeau's full generalized numbers is defined as 
$\overline{\Kset}_f=\mathcal{E}_f^M(\Kset)/\mathcal{N}_f(\Kset)$.}}
\end{Def}

There exists a natural embedding of $\Kset$ into $\overline{\Kset}_f$
(induced by the map $v\mapsto (\varphi\mapsto v)$ and so 
write $\Kset\subset\overline{\Kset}_f$ makes sense. 
Hence $\overline{\Kset}_f$ is a unital commutative $\Kset$-algebra. 
The following definition is well known (see \cite{MK}).

\begin{Def} {\tmsamp{An element $v\in\overline{\Kset}_f$ is associated 
to zero, $v\approx 0$, if for some (hence for each) representative 
$(v(\varphi))_\varphi$ of $v$ we have: 
$$\exists ~p\in\Nset ~\mbox{such that}~\lim_{\varepsilon\downarrow
  0}v(\varphi_\varepsilon)=0,
  ~\forall~\varphi\in\mathcal{A}_p(\Kset).$$
Two elements $v_1,v_2\in\overline{\Kset}_f$ are associated, $v_1\approx
v_2$, if $v_1-v_2\approx 0$. If there exists some $a\in\Kset$ with 
$v\approx a$, then $a$ is called associated the number or shadow of $v$.}}   
\end{Def} 

We denote by $\Inv({\overline{\Kset}_f})$ the set of units of
$\overline{\Kset}_f$ and clearly $\Kset^*$ is a subgroup of the
multiplicative group $\Inv({\overline{\Kset}_f})$. Another interesting
subgroup of $\Inv({\overline{\Kset}_f})$ is
$H:=\{\dot{\alpha}_r|r\in\Rset\}$, where
$\hat{\dot{\alpha}}_r(\varphi):=\diam(\supp(\varphi))^r$ or 
$\hat{\dot{\alpha}}_r(\varphi):=(i(\varphi))^r$, where $i$ denotes the
diameter of $\varphi\in\mathcal{A}_0(\Kset)$. In particular, it is
easy to see that
\begin{equation}\label{artigo-1}
\hat{\dot{\alpha}}_r (\varphi_{\varepsilon})={\varepsilon}^r(i(\varphi))^r.
\end{equation}

For $r\in\Rset, ~\hat{\alpha}_r:]0,1]\to \Rset_+$ is defined by 
$\hat{\alpha}_r(\varepsilon):=\varepsilon^r$. It is convenient to 
define $\dot{\beta}_r:=\dot{\alpha}_{-\log(r)}$.
 By (\ref{artigo-1}) we have that 
$\hat{\dot{\alpha}}_r(\varphi_\varepsilon)=\hat{\alpha}_r(\varepsilon)(i(\varphi))^r.$

We now briefly describe the sharp topology on $\overline{\Kset}_f$ 
(see \cite{NTCA}). 

\begin{Def}\label{valuacao}{\tmsamp{ 
For a given $x\in\overline{\Kset}_f$ we set
$$\A(x):=\{r\in\Rset|\dot{\alpha}_rx\approx 0\}$$ and define the 
``valuation'' of $x$ by $$\V(x):=\sup(\A(x)).$$}}
\end{Def}
 
The following results are easily deduced.
\begin{Lem}\label{novo} {\tmsamp{Let $x\in\overline{\Kset}_f$. 
Then $r\in\A(x)$ if, and only if, there
exists $p\in\Nset$ such that $$\lim_{\varepsilon\downarrow
0}\varepsilon^{-r}x(\varphi_\varepsilon)=0, 
~\forall~\varphi\in\mathcal{A}_p(\Kset).$$}}
\end{Lem}
 
\begin{Pro}\label{valor}{\tmsamp{For $x,y\in\mathcal{E}_f^M(\Kset)$ 
and $\lambda\in\Kset$, we have:
\begin{enumerate}
\item[$(i)$] $\V(\lambda x)=\V(x)$ if $\lambda\ne 0$;
\item[$(ii)$] $\V(xy)\ge\V(x)+\V(y)$;
\item[$(iii)$] $\V(\dot{\alpha}_r x)=r+\V(x)$ for each $r\in\Rset$;
\item[$(iv)$] $\V(x+y)\ge\inf\{\V(x),\V(y)\}$;
\item[$(v)$] $\V(x)=+\infty \Leftrightarrow x\in\mathcal{N}_f(\Kset)$;
\item[$(vi)$] $\V$ is constant on each equivalence class modulo
$\mathcal{N}_f(\Kset)$, i.e., $\V(x+y)=\V(x),
~\forall~x\in\mathcal{E}_f^M(\Kset)$ and $y\in\mathcal{N}_f(\Kset)$.
\end{enumerate}}}
\end{Pro}
 
It follows that
$D:\overline{\Kset}_f\times\overline{\Kset}_f\to\Rset_+$ defined by
$$D(x,y):=\exp(-\V(x-y))$$ is a metric. In fact it is an ultra-metric on
$\overline{\Kset}_f$ invariant under translations.
 
$D$ determines a uniform structure on $\overline{\Kset}_f$ called the 
{\it{sharp uniform structure}} on $\overline{\Kset}_f$ and the
topology resulting from $D$ is called the {\it{sharp topology}} 
on $\overline{\Kset}_f$ and denoted by $\tau_{sf}$.

\begin{Not} {\tmsamp{Let $x\in\overline{\Kset}_f$ and $r\in\Rset_+^*$. In what
follows $B_r(x)$ (resp. $B'_r(x)$ and $S_r(x)$) denotes the open
$D$-ball (resp. closed $D$-ball and $D$-sphere) with center at $x$
and radius $r$. In case $x=0$ we  omit it in the
notation, writing $B_r, B'_r$ and $S_r$. For the sake of simplicity we
define $\|x\|:=D(x,0), ~\forall~x\in\overline{\Kset}_f$ and the
distance between two elements $x,y\in\overline{\Kset}_f$ by
$d(x,y):=\|x-y\|$.}}  
\end{Not}

The following result is a direct consequence of Proposition \ref{valor}
and of the definition of $D$.
 
\begin{Cor}\label{norma} {\tmsamp{For given $x,y\in\overline{\Kset}_f,
r\in\Rset, s\in\Rset_+^*$ and $a,b\in\Kset$ we have:
\begin{enumerate}
\item[$(i)$] $\|x+y\|\le\max\{\|x\|,\|y\|\}$ and $\|xy\|\le\|x\|\|y\|$;
\item[$(ii)$] $\|x\|=0\Leftrightarrow x=0$;
\item[$(iii)$] $\|ax\|=\|x\|$, if $a\ne 0$;
\item[$(iv)$] $\|\dot{\alpha}_rx\|=e^{-r}\|x\|$ and $\|\beta_sx\|=s\|x\|$;
\item[$(v)$] $\|a\|=1$, if $a\ne 0$;
\item[$(vi)$] $\|a-b\|=1-\delta_{ab} ~(\mbox{Kronecker}~\delta)$.  
\end{enumerate}}}
\end{Cor}

\begin{Pro}\label{complete} {\tmsamp{$\left(\overline{\Kset}_f,\tau_{sf}\right)$ is a complete topological ring.}}
\end{Pro}
\begin{proof} It follows from the completude of algebras
  $\mathcal{G}\left(\overline{\Omega}\right)$ and
  $\mathcal{G}(\Omega)$ (see \cite{AGJ}) and from the fact that
  $\overline{\Kset}_f$ is the ring of the constants of such algebras. 
\end{proof}

\begin{Pro}\label{naoe}{\tmsamp{$\left(\overline{\Kset}_f,\tau_{sf}\right)$ 
is not: 
\begin{enumerate} 
\item[$(i)$] a $\Kset$-topological vector space;
\item[$(ii)$] separable;
\item[$(iii)$] locally compact.  
\end{enumerate}}}
\end{Pro}
\section{Algebraic properties of $\overline{\Kset}_f$}
\label{sec-02}
In this section we study some algebraic properties of
$\overline{\Kset}_f$. $\overline{X}$ stand for the topological 
closure of the set $X$ (except in the notation $\overline{\Kset}_f$). 
We start with the following:

\begin{Lem}\label{boll}{\tmsamp{
\begin{enumerate}
\item[$(i)$] $x\in B_1$ iff $\V(x)>0$.
\item[$(ii)$] If $x\in B_1$, then the following statements hold:
\begin{enumerate}
\item[$(a)$] $x\approx 0$;
\item[$(b)$] $D(1,x)=\|1-x\|\ge 1$. Hence $1\notin \overline{B}_1,
    B_1\cap B_1(1)=\emptyset, B'_1\supset\overline{B}_1$ and
    $B'_1\ne\overline{B}_1$.   
\end{enumerate}
\end{enumerate}}}
\end{Lem} 

\begin{Pro}\label{inv} {\tmsamp{Let $\{a_n\}_{n\in\Nset}$ be any sequence in 
$\Kset$ and let $x\in B_1$. Then the series $\underset{n\ge 0}{\sum}a_nx^n$
  converges in $\overline{\Kset}_f$. In particular $\underset{n\ge
    0}{\sum}x^n$ converges and we have $(1-x)\left(\underset{n\ge
    0}{\sum}x^n\right)=1$. Therefore $(1-x)\in\Inv(\overline{\Kset}_f)$.}}
\end{Pro}

\begin{proof} The proof is standard an left for the reader.
\end{proof}

\begin{Cor}\label{aberto}
  {\tmsamp{$\Inv(\overline{\Kset}_f)\subset\overline{\Kset}_f$ is
      open.}}
\end{Cor}

\begin{proof} We need to prove that for all 
$a\in\Inv(\overline{\Kset}_f)$ there exists $r>0$ such that $B_r(a)\subset
\Inv(\overline{\Kset}_f)$. Let $a\in\Inv(\overline{\Kset}_f)$, 
defines $r:=\|a^{-1}\|^{-1}$ and let $z\in B_r(a)$. Then $\|a^{-1}(z-a)\| \leq \|a^{-1}\|\|z-a\|
< r^{-1}r =1$. Therefore $\|a^{-1}(z-a)\|<1$ and hence, by Proposition
\ref{inv}, we have that
$(1-a^{-1}(a-z))\in\Inv(\overline{\Kset}_f)$. Since 
$z=a(1-a^{-1}(a-z))$ it follows that $z\in\Inv(\overline{\Kset}_f)$.
\end{proof}

\begin{The}\label{ole-1} {\tmsamp{Let $\mathfrak{I}$ be a proper ideal of
  $\overline{\Kset}_f$. Then for each $x\in\mathfrak{I}$ we have that 
  $D(1,x)\geq 1$ and $D(1,\mathfrak{I})=1$. Hence every maximal ideal
  of $\overline{\Kset}_f$ is closed and thus it is also a rare set.}}
\end{The}

\begin{proof} If $x\in\mathfrak{I}$ then  
$x\notin\Inv(\overline{\Kset}_f)$ and so $D(1,x)=\|1-x\|\ge 1$. 
If $D(1,x)=\|1-x\|< 1$ then, from Proposition \ref{inv}, we have 
that $1-(1-x)=x\in\Inv(\overline{\Kset}_f)$, a contradiction. Since 
$D(1,0)=\|1\|=1$, the first part is proved.  
If $\mathfrak{m}$ is a maximal ideal of $\overline{\Kset}_f$ then, 
from the first part, we have that $1\notin\overline{\mathfrak{m}}$ 
and we are done.
\end{proof}

\begin{Lem}\label{feixo}{\tmsamp{
\begin{enumerate}
\item[$(i)$] $0\in\overline{\Inv(\overline{\Kset}_f)}$.
\item[$(ii)$] For each $x\in\overline{\Kset}_f$, if $r=-\V(x)$, then
  $y=\dot{\alpha}_rx\in S_1$, i.e., $\|y\|=1$.
\item[$(iii)$] If $x\in\Inv(\overline{\Kset}_f)$, then $\V(x)+\V(x^{-1})\le 0$. 
\end{enumerate}}}
\end{Lem}

\begin{proof} The proof is as in \cite{AJ}.
\end{proof}

Another results that is analogous to one of \cite{AJ} is

\begin{Pro}\label{ole-2}{\tmsamp{
\begin{enumerate}
\item[$(i)$] $\overline{\Kset}_f$ does not have proper open ideals.
\item[$(ii)$] No topological $\overline{\Kset}_f$-modules have proper
  open submodules.
\item[$(iii)$] If $X$ is a Hausdorff topological
  $\overline{\Kset}_f$-modules then for all $x\in X, ~x\ne 0$, the set
  $$\Inv(\overline{\Kset}_f).x:=\{\lambda
  x|\lambda\in\Inv(\overline{\Kset}_f)\}$$ is not a bounded
  set. Whence, $\Inv(\overline{\Kset}_f)$ is not a bounded subset of
  $\overline{\Kset}_f$.
\item[$(iv)$] A given $B\subset\overline{\Kset}_f$ is bounded iff $B$
  is $D$-bounded.
\item[$(v)$] The only $\overline{\Kset}_f$-topological module which is
  bounded is $\{0\}$. Whence, the only
  $\overline{\Kset}_f$-topological module which is compact is $\{0\}$ 
\end{enumerate}}}
\end{Pro}
 
\begin{The}\label{ident}{\tmsamp{Let $\mathfrak{m}$ be a maximal ideal 
of $\overline{\Kset}_f$ and $L:=\overline{\Kset}_f/\mathfrak{m}$. 
Then  $\Kset$ can be identified with a proper subfield of $L$, i.e., $L$
is a proper field extension of $\Kset$.}}
\end{The}

\begin{proof} We follows the proof of \cite{AJ}. 
Let $\pi:\overline{\Kset}_f\to L:=\overline{\Kset}_f/\mathfrak{m}$ 
be the canonic application. Then
 $k:=\pi(\Kset)\simeq \Kset$. In fact, if $L=k$, then
$\overline{\Kset}_f=\Kset+\mathfrak{m}$. But $\Kset$ is a discrete
subset of $\overline{\Kset}_f$ and hence, from Theorem \ref{ole-1}
and the Proposition \ref{ole-2} $(i)$, it follows immediately that
$\mathfrak{m}\cup \Kset$ is a closed set with empty interior. 
Thus, there exists $x\in B_1$, such that
$x\notin\mathfrak{m}\cup\Kset$. Write $x=k_x-m_x$, where
$k_x\in\Kset$ and $m_x\in\mathfrak{m}$. Obviously $k_x\ne 0$ and hence
$m_x=k_x-x=k_x(1-k_x^{-1}x)$. Indeed, 
$(1-k_x^{-1}x)=k_x^{-1}m_x\in\mathfrak{m}$ and so 
$(1-k_x^{-1}x)\in\mathfrak{m}$. Since $\|k_x^{-1}x\|=\|x\|<1$,it
follows that $m_x$ is a unit, a contradiction.
\end{proof}

\begin{Lem}\label{fato} {\tmsamp{Let $R_1, R_2$ be positive real
numbers and set $r:=\ln(R_1)-\ln(R_2)$. Then
$\dot{\alpha}_r.S_{R_1}=S_{R_2}$. In particular, if 
$x\in\overline{\Kset}_f\setminus\{0\}$ and $r=-\V(x)$ then 
$\dot{\alpha}_rx\in S_1$.}} 
\end{Lem}
\begin{proof} This is an immediate consequence of Corollary 
\ref{norma} $(iv)$.
\end{proof}

Another easily deducible fact is that $\overline{\Kset}_f$ has no 
non-zero nilpotent elements and hence its nil-radical is trivial. 
Consequently $\overline{\Kset}_f$ is contained in a product of
integral domains.
\section{Characteristic functions}\label{sec-03} 
In this section we study the group of the units of
the ring of Colombeau's full generalized numbers
$\overline{\Kset}_f$, as well as its prime and maximal ideals. 
The crucial step in this study is to make a careful analysis 
of the set of zeros of a representative of an element of 
$\overline{\Kset}_f$. Using this we study some special type 
of characteristic functions showing that they are
related with the prime and maximal ideals of $\overline{\Kset}_f$. 
One of the main consequence of this analysis is that the set of 
the unit group of $\overline{\Kset}_f$ is open and dense in the 
sharp topology. The new feature here is the density since we have 
already proved it to be open. The idea behind the proof is closely 
related to Aragona and Juriaans paper (\cite{AJ}). However there 
is a crucial difference.

It is useful for the reader to recall the definition of the sets 
$\mathcal{S}$ and $\mathit{P}_*(\mathcal{S}_f)$ defined in
\cite{AJ} (Definition 4.1 pg. 2217). We will use this notation in 
what follows. For $A\subset\mathcal{A}_0(\Kset)$ let $A^c$ be its 
complement in $\mathcal{A}_0(\Kset)$ and denote by
$\hat{\mathcal{X}}_A$ the characteristic function of $A$ with domain 
$\mathcal{A}_0(\Kset)$, i.e.,
\begin{displaymath}
\hat{\mathcal{X}}_A(\varphi)=\left\{\begin{array}{ll}
1, & \mbox{se $\varphi\in A$}\\
0, & \mbox{se $\varphi\notin A$, i.e., $\varphi\in A^c$}.
\end{array}\right.
\end{displaymath}

This is clearly a moderate function. Its class in $\overline{\Kset}_f$ 
is denoted  by $\mathcal{X}_A$ and still called the characteristic 
function of $A\subset\mathcal{A}_0(\Kset)$.
\begin{Def} {\tmsamp{Define
  $$\mathcal{S}_f:=\{A\subset\mathcal{A}_0(\Kset)|\forall~p\in\Nset,
  \exists~\varphi\in\mathcal{A}_p(\Kset)~\mbox{such
    that}~\{\varepsilon|\varphi_\varepsilon\in A\}\in \mathcal{S}\},$$
  where $\mathcal{S}:=\{S\subset I|0\in\overline{S}\cap\overline{S^c}\}$.}}
\end{Def}

\begin{Pro}\label{caos}{\tmsamp{ 
\begin{enumerate}
\item[$(i)$] $A\in\mathcal{S}_f\Leftrightarrow A^c\in\mathcal{S}_f$. 
\item[$(ii)$] $A\in\mathcal{S}_f\Rightarrow \emptyset\ne
  A\cap\mathcal{A}_p(\Kset)\subsetneqq \mathcal{A}_p(\Kset), ~\forall
  ~p\in\Nset$. In particular, we have
  $A\cap\mathcal{A}_p(\Kset)\ne\emptyset$ and
  $A^c\cap\mathcal{A}_p(\Kset)\ne\emptyset$ for all $p\in\Nset$.
\end{enumerate}}}
\end{Pro}

\begin{proof} $(i)$ Indeed, if $A\in\mathcal{S}_f$, then for all
$p\in\Nset$ there 
exists $\varphi\in\mathcal{A}_p(\Kset)$ such that
$\{\varepsilon|\varphi_\varepsilon\in A\}\in\mathcal{S}$. Thus, from 
(Proposition 4.2 item b pg. 2217 of \cite{AJ}), we
have that
$\{\varepsilon|\varphi_\varepsilon\in A\}^c\in\mathcal{S}$; but
$\{\varepsilon|\varphi_\varepsilon\in
A\}^c=\{\varepsilon|\varphi_\varepsilon\notin
A\}=\{\varepsilon|\varphi_\varepsilon\in A^c\}$ and therefore 
$\{\varepsilon|\varphi_\varepsilon\in A^c\}\in\mathcal{S}$. Hence for
all $p\in\Nset$ there exists $\varphi\in\mathcal{A}_p(\Kset)$ such that
$\{\varepsilon|\varphi_\varepsilon\in A^c\}\in\mathcal{S}$ and,
therefore, $A^c\in\mathcal{S}_f$. Reciprocally, if
$A^c\in\mathcal{S}_f$, then $A=(A^c)^c\in\mathcal{S}_f$.

$(ii)$ If $A\in\mathcal{S}_f$,
then for all $p\in\Nset$, there exists
$\varphi\in\mathcal{A}_p(\Kset)$ such that
the set $A_\varepsilon:=\{\varepsilon|\varphi_\varepsilon\in
A\}\in\mathcal{S}$. As
$\{\varphi_\varepsilon|\varepsilon\in
A_\varepsilon\}\subseteq\mathcal{A}_p(\Kset)$ and
$\{\varphi_\varepsilon|\varepsilon\in A_\varepsilon\}\subseteq A$,it 
follows that $A\cap\mathcal{A}_p(\Kset)\ne\emptyset$. 
We now prove that $A\cap\mathcal{A}_p(\Kset)\subsetneqq
\mathcal{A}_p(\Kset)$. That it is contained is obvious. Now that
$A\cap\mathcal{A}_p(\Kset)\ne\mathcal{A}_p(\Kset)$ 
follows from the fact that $\{\varphi_\varepsilon|\varepsilon\in
A_\varepsilon^c\}\subset\mathcal{A}_p(\Kset)$ 
and $\{\varphi_\varepsilon|\varepsilon\in
A_\varepsilon^c\}\subset A^c \in \mathcal{S}_f$.
This implies that $A^c\cap\mathcal{A}_p(\Kset)\ne\emptyset$. Hence, 
$A\cap\mathcal{A}_p(\Kset)\subsetneqq \mathcal{A}_p(\Kset)$.
\end{proof}

We have the following obvious result.
\begin{Pro}\label{direta} {\tmsamp{Let $A\in\mathcal{S}_f$. 
Then $\Ann(\mathcal{X}_A)=\overline{\Kset}_f\mathcal{X}_{A^c}$ and
$\overline{\Kset}_f=\Ann(\mathcal{X}_A)\oplus\Ann(\mathcal{X}_{A^c})$.
Moreover, for each prime ideal $\mathfrak{p}$ of $\overline{\Kset}_f$
we have that either $\mathcal{X}_A$ or
$\mathcal{X}_{A^c}=1-\mathcal{X}_A$ belongs to $\mathfrak{p}$.}} 
\end{Pro}

\begin{Not} {\tmsamp{In what follows the symbol $\mathit{P}(\mathcal{S}_f)$
denotes the set of all subsets of $\mathcal{S}_f$.}} 
\end{Not}

\begin{Def}\label{def} {\tmsamp{We denote by
      $\mathit{P}_*(\mathcal{S}_f)$ 
the set of all
  $\mathcal{F}\in \mathit{P}(\mathcal{S}_f)$ verifying the 
  following conditions: 
\begin{enumerate}
\item[$(i)$] For every $A\in\mathcal{S}_f$, either $A$ or $A^c$
  belongs to $\mathcal{F}$ but not both.
\item[$(ii)$] If $A,B\in \mathcal{F}$ then $A\cup B\in\mathcal{F}$.
\end{enumerate}}}
\end{Def}

The second condition says that $\mathcal{F}$ is stable under finite
union. The following proposition enumerates some properties of the
functions $\mathcal{X}_A$. 
\begin{Pro}\label{cara}{\tmsamp{
\begin{enumerate}
\item[$(i)$] If $A\in\mathcal{S}_f$ then $\mathcal{X}_A\in S_1$. 
\item[$(ii)$] If $A, B\in \mathcal{F}, ~A\ne B$ and
    $\mathcal{X}_A\ne\mathcal{X}_B$ then 
    $\dis(\mathcal{X}_A,\mathcal{X}_B)=\|\mathcal{X}_A-\mathcal{X}_B\|=1$. 
Hence the topology of $\overline{\Kset}_f$ does not have an enumerable base.
 \item[$(iii)$] If $A\in\mathcal{S}_f$ then
    $\mathcal{X}_A\in\overline{\Kset}_f\setminus\{0,1\}$ and 
$\mathcal{X}_A=\mathcal{X}_A^2$. 
\item[$(iv)$] $(1-\mathcal{X}_A)\mathcal{X}_A=0$.
\end{enumerate}}}
\end{Pro}

\begin{proof} $(i)$ To show that $\|\mathcal{X}_A\|=1$
observe that $\dot{\alpha}_{-r}\approx 0\Leftrightarrow r<0.$ (see \cite{NTCA})
Thus, if $r\in \A(\mathcal{X}_A)$ (see Definition \ref{valuacao}), then
$\dot{\alpha}_{-r}\mathcal{X}_A\approx 0$ and, therefore,
$\dot{\alpha}_{-r}\approx 0$ and so $A(\mathcal{X}_A)=]-\infty, 0[$. 
Hence  
$\|\mathcal{X}_A\|=e^{-\V(\mathcal{X}_A)}=e^{-\sup(A(\mathcal{X}_A))}=e^0=1.$

$(ii)$ Let $r\in A(\mathcal{X}_A-\mathcal{X}_B)$. 
Then $\dot{\alpha}_{-r}(\mathcal{X}_A-\mathcal{X}_B)\approx 0$. 
As  $\mathcal{X}_A\ne \mathcal{X}_B$ it follows that 
$\dot{\alpha}_{-r}\approx 0$ and 
$A(\mathcal{X}_A-\mathcal{X}_B)=]-\infty, 0[$. Hence, 
$\|\mathcal{X}_A-\mathcal{X}_B\|=e^{-\V(\mathcal{X}_A-\mathcal{X}_B)}
=e^{-\sup(A(\mathcal{X}_A-\mathcal{X}_B))}=e^0=1.$

The $(iii)$ and $(iv)$ are obvious. 
\end{proof}

\begin{Def}\label{degf} {\tmsamp{For 
$\mathcal{F}\in\mathit{P}_*(\mathcal{S}_f)$ let 
  $g_f(\mathcal{F}):=\langle\{\mathcal{X}_A|A\in\mathcal{F}\}\rangle$
be the ideal of $\overline{\Kset}_f$ generated by the set of the
characteristic functions of $A$ with $A\in\mathcal{F}$.}} 
\end{Def}

The following is an easy to result.
\begin{Lem} {\tmsamp{If $\mathcal{F}\in\mathit{P}_*(\mathcal{S}_f)$ and
  $A,B\in\mathcal{F}$ then $\mathcal{X}_{A\cap B}, \mathcal{X}_{A\cup
    B}\in g_f(\mathcal{F})$.}}
\end{Lem}
 
\begin{Lem}\label{idpro} {\tmsamp{If 
$\mathcal{F}\in \mathit{P}_*(\mathcal{S}_f)$ then
$g_f(\mathcal{F})$ is a proper ideal of $\overline{\Kset}_f$, i.e.,
  $g_f(\mathcal{F})\ntrianglelefteq\overline{\Kset}_f$.}}
\end{Lem}

\begin{proof} Assume that $g_f(\mathcal{F})=\overline{\Kset}_f$. Then
there exists $a_i\in\overline{\Kset}_f$ and 
$A_i\in\mathcal{F}, ~i=1,2,\dots, n$ such that 
$1=\sum_{i=1}^na_i\mathcal{X}_{A_i}$. We define $A:=\bigcup_{i=1}^n
A_i$. Then $\mathcal{A}_0(\Kset)\ne A\in\mathcal{F}$ 
and $\forall i=1,2,\dots, n ~\mathcal{X}_{A_i}\cdot\mathcal{X}_A=
\mathcal{X}_{A_i\cap A}=\mathcal{X}_{A_i}.$ 
Multiplying both members of the former equation by 
$\mathcal{X}_A$, we obtain  $\mathcal{X}_A=\sum_{i=1}^n
a_i\mathcal{X}_{A_i}.\mathcal{X}_A=\sum_{i=1}^n
a_i\mathcal{X}_{A_i}=1.$ Hence 
$(\hat{\mathcal{X}}_A-1)\in\mathcal{N}_f(\Kset)$. 
It follows that $\exists~ p\in\Nset, ~\exists~\gamma\in\Gamma$ 
such that $\forall~q\ge p ~\mbox{e}~\forall~\varphi\in\mathcal{A}_q(\Kset)
~(\exists~ C=C_\varphi>0, ~\exists~\eta=\eta_\varphi>0$) such that
$|1-\hat{\mathcal{X}}_A(\varphi_\varepsilon)| \le C\varepsilon^{\gamma(q)-p},
~\forall~0<\varepsilon<\eta$. But
$A\in\mathcal{F}\subset\mathit{P}(\mathcal{S}_f)$ and so, for this
$\varphi$, $\{\varepsilon|\varphi_\varepsilon\in
A\}\in\mathcal{S}\Rightarrow 0\in\overline{\{\varepsilon|\varphi_\varepsilon\in
A\}}\cap\overline{\{\varepsilon|\varphi_\varepsilon\in A^c\}},$ 
i.e., there exists a sequence $\{\varepsilon_n\}_{n\in\Nset}$ 
converging to zero when $n\to\infty$ such that 
$\varphi_{\varepsilon_n}\in A^c$. It now follows that   
$1=|1-\hat{\mathcal{X}}_A(\varphi_{\varepsilon_n})|\le 
C(\varepsilon_n)^{\gamma(q)-p}$.
As $\gamma$  is divergent, we can choose $q_0\in\Nset$ such that
$\gamma(q_0)-p> 0$. Thus, we have $1\le
C(\varepsilon_n)^{\gamma(q_0)-p}\underset{n\longrightarrow\infty}{\rightarrow}
0$, a contradiction. Therefore, 
$g_f(\mathcal{F})\ntrianglelefteq\overline{\Kset}_f$. 
\end{proof}

\begin{The}\label{prime} {\tmsamp{Let $\mathfrak{p}$ be a prime ideal of
  $\overline{\Kset}_f$. Then there exists a unique  
  $\mathcal{F}_\mathfrak{p}\in\mathit{P}_*(\mathcal{S}_f)$ such that
  $g_f(\mathcal{F}_\mathfrak{p})\subset\mathfrak{p}$. In particular, 
$\mathit{P}_*(\mathcal{S}_f)\ne\emptyset$.}}
\end{The}

\begin{proof} The proof is similar to that given in \cite{AJ}.
\end{proof}

Theorem \ref{prime} associates with each prime ideal
$\mathfrak{p}$ of $\overline{\Kset}_f$ a set
$\mathcal{F}_\mathfrak{p}\in\mathit{P}_*(\mathcal{S}_f)$ characterized
by the inclusion $g_f(\mathcal{F}_\mathfrak{p})\subset\mathfrak{p}$.

In what follows $\overline{g_f(\mathcal{F})}$, where
$\mathcal{F}\in\mathit{P}_*(\mathcal{S}_f)$, denotes the
$\tau_{sf}$-closure of $g_f(\mathcal{F})$ in $\overline{\Kset}_f$.  

\begin{Def} {\tmsamp{Let $x\in\overline{\Kset}_f$ and $\hat{x}$ one of
its representative. We define 
$Z(\hat{x}):=\{\varphi\in\mathcal{A}_0(\Kset)|\hat{x}(\varphi)=0\}$, 
the set of zeros of the representative $\hat{x}$ of $x$.}}
\end{Def}

\begin{Lem} {\tmsamp{Let $x\in\overline{\Kset}_f\setminus\{0\}$ and
$\mathcal{F}\in\mathit{P}_*(\mathcal{S}_f)$. Then the following
statements are equivalents: 
\begin{enumerate}
\item[$(i)$] $x\in g_f(\mathcal{F})$.
\item[$(ii)$] There exists $A\in\mathcal{F}$ such that $x\mathcal{X}_A=x$.
\end{enumerate}}}
\end{Lem}

\begin{proof} $(i)\Rightarrow (ii)$: If $x\in g_f(\mathcal{F})$,
then there exists $A_1,A_2,\dots, A_n\in\mathcal{F}$ and $a_1,a_2,\dots,
a_n\in\overline{\Kset}_f$ such that $x=\sum_{i=1}^n
a_i\mathcal{X}_{A_i}.$ Definition \ref{def} $(ii)$ tells us that 
that $A:=\bigcup_{i=1}^n A_i\in\mathcal{F}$ and since $A_i\subset A,
~\forall~ i=1,2,\dots, n$ it follows that
$\mathcal{X}_{A_i}\cdot\mathcal{X}_A=\mathcal{X}_{A_i}$. Hence 
$  x\mathcal{X}_A=\left(\sum_{i=1}^n a_i\mathcal{X}_{A_i}\right)\mathcal{X}_A
=\sum_{i=1}^n (a_i\mathcal{X}_{A_i})\mathcal{X}_A
=\sum_{i=1}^n a_i(\mathcal{X}_{A_i}\mathcal{X}_A)
=\sum_{i=1}^n a_i\mathcal{X}_{A_i}=x$.
Therefore, $x\mathcal{X}_A=x$ and $A\in\mathcal{F}$.

$(ii) \Rightarrow (i)$: If $A\in\mathcal{F}$ and (see \cite{NTCA})
$x\mathcal{X}_A=x$, then $x\in g_f(\mathcal{F})$. Therefore there exists
$A\in\mathcal{F}$ and $x\in\overline{\Kset}_f$ such that 
$x=x\mathcal{X}_A$, where $\mathcal{X}_A\in g_f(\mathcal{F})$.
\end{proof}

\begin{The}{\tmsamp{
\begin{enumerate}
\item[$(i)$] $x\in\Inv(\overline{\Kset}_f)$ if and only if 
    $Z(\hat{x})\notin \mathcal{S}_f, ~\forall$ representative $\hat{x}$
    of $x$.
\item[$(ii)$] $x\notin\Inv(\overline{\Kset}_f)$ if and only if $\exists~
  e\in\overline{\Kset}_f, ~e^2=e$ such that $x\cdot e=0$. In particular, if
  $x\in\overline{\Kset}_f\setminus\{0\}$ and $x\notin\Inv(\overline{\Kset}_f)$,
  then $x$ is a zero divisor. 
\end{enumerate}}}
\end{The}

\begin{proof} $(i)$ Suppose by contradiction that
$x\in\Inv(\overline{\Kset}_f)$ and that $Z(\hat{x})\in\mathcal{S}_f$
for some representative $\hat{x}$ of $x$. Then, by Proposition
\ref{cara} $(iii)$, we have that
$\hat{\mathcal{X}}_{Z(\hat{x})}\notin\{0,1\}$. Thus
$\hat{\mathcal{X}}_{Z(\hat{x})}\ne 0$ but
$x\mathcal{X}_{Z(\hat{x})}=0$ and hence $x$ is a zero divisor,
a contradiction. 

To prove the converse it is enough to show ($ii$). 

($ii$) We will show that if 
$x\notin\Inv(\overline{\Kset}_f)$, then $x$ is a zero divisor. 
In fact let $\hat{x}$ be a representative of $x$.  
We consider two cases:\\
\begin{enumerate}
\item[($a$)] $Z(\hat{x})\in\mathcal{S}_f$.
\item[($b$)] $Z(\hat{x})\notin\mathcal{S}_f$.
\end{enumerate}

($a$) If $Z(\hat{x})\in\mathcal{S}_f$ then, by Proposition \ref{cara} 
$(iii)$, we have that $\hat{\mathcal{X}}_{Z(\hat{x})}\ne 0 $ and since 
$x\mathcal{X}_{Z(\hat{x})}=0$ we have that $x$ is a zero divisor. 
($b$) If $Z(\hat{x})\notin\mathcal{S}_f$  define
$x^*(\varphi)=\left\{\begin{array}{ll} \hat{x}(\varphi),
      &\mbox{if}~\varphi\notin Z(\hat{x})\\
0, &\mbox{if}~\varphi\in Z(\hat{x})\end{array}\right.
$, i.e., $x^*={\hat{x}}\mathcal{X}_{Z(\hat{x})^c}$. Then 
$x^*(\varphi)\ne 0$ for some $\varphi\in\mathcal{A}_0(\Kset)$ 
and $(x^*-\hat{x})\in\mathcal{N}_f(\Kset)$. So we may substitute 
$\hat{x}$ by $x^*$ and assume that $\hat{x}(\varphi)\ne 0$. 

Since $x\notin\Inv(\overline{\Kset}_f)$ it follows that
$\frac{1}{x}\notin\mathcal{E}_f^M(\Kset)$ and hence, by  
definition \ref{kset-1},  $\forall~p\in\Nset,
~\exists~\varphi\in\mathcal{A}_p(\Kset)~\mbox{such that }~\forall~C
=C_\varphi>0,~\forall~\eta=\eta_\varphi>0,
~\exists~0<\varepsilon_0<\eta~\mbox{such that }~
\left|\frac{1}{\hat{x}}(\varphi_{\varepsilon_0})\right|>C\varepsilon_0^{-p}$. 
Taking $C=n, ~\eta=\frac{1}{n}$, we have that there exists 
$0<\varepsilon_n<\frac{1}{n}$ such that 
$$\left|\frac{1}{\hat{x}}(\varphi_{\varepsilon_n})\right|>n\varepsilon_n^{-p}.$$
For this $\varphi$ we define the set
$A_p(\varphi_p):=\{\varepsilon_n|n\in\Nset\}$ and let 
$B:=\{(\varphi_p)_{\varepsilon_n}|\varepsilon_n\in
A_p(\varphi_p)\}.$ Then $B\in\mathcal{S}_f$, i.e.,
$\forall~p\in\Nset, ~exists~\varphi_p\in\mathcal{A}_p(\Kset)~\mbox{tal
  que}~\{\varepsilon_n|(\varphi_p)_{\varepsilon_n}\in
B\}\in\mathcal{S}$. Indeed, it is enough to show that 
$\{\varepsilon_n|(\varphi_p)_{\varepsilon_n}\in B\}\in\mathcal{S}$. 
For that is enough to notice that
$A_p(\varphi_p)=\{\varepsilon_n|(\varphi_p)_{\varepsilon_n}\in B\}$ and 
therefore $\emptyset\ne A_p(\varphi_p)\cap I_{\eta=1/2}\ne
I_{\eta=1/2}$ 
because $\lim_{n\to\infty}\varepsilon_n\le\lim_{n\to\infty}\frac{1}{n}=0$.
 
We now show that
  $\hat{x}\hat{\mathcal{X}}_B\in\mathcal{N}_f(\Kset)$. Definition
  \ref{kset} tells us that we need to prove that
$\exists~p\in\Nset
  ~\exists~\gamma\in\Gamma~\mbox{tal que}~\forall~q\ge p,~\forall
  ~\varphi\in\mathcal{A}_q(\Kset)~\exists~C=C(\varphi)>0,
  ~\exists~\eta=\eta(\varphi)>0~\mbox{tal que}~
  |\hat{x}\hat{\mathcal{X}}_B(\varphi_\varepsilon)|\le C\varepsilon^{\gamma(q)-p},
  ~\forall~ 0<\varepsilon<\eta$. In fact, taking $p=0$ and $
  ~\gamma(q)=q$ we have that $|\hat{x}\hat{\mathcal{X}}_B
(\varphi_\varepsilon)|\le C\varepsilon^q<1$ for small $\varepsilon$. 
Therefore, $\hat{x}\hat{\mathcal{X}}_B\in \mathcal{N}_f(\Kset)$.
\end{proof}

\begin{Def}\label{nax}{\tmsamp{
For $x\in\mathcal{E}_f^M(\Kset)$ and $a\in\Nset$ define
\begin{enumerate}
\item[$(i)$] $N_a(x):=\{\varphi\in\mathcal{A}_0(\Kset)|
|x(\varphi)|<\dot{\alpha}_a(\varphi)\}$;
\item[$(ii)$] $\hat{\mathcal{X}}_{a,x}:=\hat{\mathcal{X}}_{N_a(x)}$
e $\mathcal{X}_{a,x}:=\mathcal{X}_{N_a(x)}$.
  \end{enumerate}}}
\end{Def}

\begin{Lem}\label{antes} {\tmsamp{If $A\subseteq \mathcal{A}_0(\Kset)$ then
  $\hat{\mathcal{X}}_A\in\mathcal{N}_f(\Kset)$ if and inly if $ ~\exists~
  \tau:\mathcal{A}_0(\Kset)\to ]0,1]$ such that for all $\varphi\in
  \mathcal{A}_q(\Kset)$ we have that 
  $I_{\tau(\varphi)}\subseteq\{\varepsilon|\varphi_\varepsilon\in
  A^c\}$.}} 
\end{Lem}

\begin{proof} We have that 
$\hat{\mathcal{X}}_A\in\mathcal{N}_f(\Kset)$ if and only if $
~\exists~ p\in\Nset, ~\exists~\gamma\in\Gamma ~\mbox{such
  that}~\forall~q\ge p,
~\forall~\varphi\in\mathcal{A}_q(\Kset)~\exists~ C=C_\varphi> 0,~ \exists
~\eta=\eta_\varphi>0 ~\mbox{such
  that} ~|\hat{\mathcal{X}}_A(\varphi_\varepsilon)|\le
C\varepsilon^{\gamma(q)-p}, ~\forall ~0<\varepsilon<\eta$. As
$\gamma$ is divergent, one can choose $q$ such that 
$\gamma(q)-p>2$. Now choose $\varepsilon$ such that 
$0<\varepsilon<\tau(\varphi)<\eta=\eta_\varphi$ and 
$C\varepsilon^{\gamma(q)-p}<1$. Under these conditions we have that
$|\hat{\mathcal{X}}_A(\varphi_\varepsilon)|<1\Leftrightarrow
\hat{\mathcal{X}}_A(\varphi_\varepsilon)=0\Leftrightarrow
\varphi_\varepsilon\in A^c\Leftrightarrow 
\{\varphi_\varepsilon|\varepsilon<\tau(\varphi)\}\subseteq
A^c$, i.e, $\{\varphi_\varepsilon|\varepsilon\in I_{\tau(\varphi)}\}
\subseteq A^c$ or equivalently 
$I_{\tau(\varphi)}\subseteq\{\varepsilon|\varphi_\varepsilon\in A^c\}$.
\end{proof}

\begin{Lem}\label{vert}{\tmsamp{ 
\begin{enumerate}
\item[$(i)$] $\mathcal{X}_{N_a(x)}=1, ~\forall~ a\in\Nset$ iff 
$x\in\mathcal{N}_f(\Kset)$.
\item[$(ii)$] $\mathcal{X}_{N_a(x)}=0, ~\forall ~a\in\Nset$ iff 
$x\in\Inv(\overline{\Kset}_f)$.  
\end{enumerate}}}
\end{Lem}

\begin{proof} $(i)$ Take $p=0$ and $ ~\gamma:\Nset\to\Nset$,
$\gamma(q)=q$. We have that $\mathcal{X}_{N_q(x)}=1$ if and only 
if $\hat{\mathcal{X}}_{(N_q(x))^c}=1-\hat{\mathcal{X}}_{N_q(x)}
\in\mathcal{N}_f(\Kset)$ and, by Lemma \ref{antes}, there exists
$\tau:\mathcal{A}_0(\Kset)\to ]0,1]$ such that 
$\forall \varphi\in\mathcal{A}_q(\Kset)$ 
$I_{\tau(\varphi)}\subseteq\{\varepsilon|\varphi_\varepsilon\in
N_q(x)\}$, i.e., $|x(\varphi_\varepsilon)|<
\dot{\alpha}_q(\varphi_\varepsilon)=(i(\varphi))^q\varepsilon^q,
~\forall~ 0<\varepsilon<\tau(\varphi)$. So if we setting 
$C=C_\varphi=(i(\varphi))^q>0, ~\eta=\eta_\varphi=\tau(\varphi)>0$ we get
$|x(\varphi_\varepsilon)|<(i(\varphi))^q\varepsilon^q=C\varepsilon^{\gamma(q)-p},
~\forall~ 0<\varepsilon<\eta(\varphi).$

Item $(ii)$ can be proved in a similar way.
\end{proof}

\begin{Pro}\label{creio} {\tmsamp{Let $x\in\overline{\Kset}_f, ~(x\ne
0)$ be a non-unit. Then there exists $a\in\Nset$ such that 
$S=N_a(x)\in\mathcal{S}_f$ and $|x\mathcal{X}_S|<\dot{\alpha}_a$.}}
\end{Pro}
\begin{proof} In the prove $x$ will stand for a representative. 
Suppose that $\forall a\in\Nset$ is we have that 
$S:=N_a(x)\notin\mathcal{S}_f$. Then either 
$\mathcal{X}_{N_a(x)}\in\mathcal{N}_f(\Kset)$ or
$\mathcal{X}_{(N_a(x))^c}\in\mathcal{N}_f(\Kset)$. In fact, if
$\mathcal{X}_{N_a(x)}\in\mathcal{N}_f(\Kset)$ then, by Lemma
\ref{antes}, there exists $\tau:\mathcal{A}_0(\Kset)\to ]0,1]$ 
such that $\ \forall \varphi\in\mathcal{A}_q(\Kset)$ we have that
$I_{\tau(\varphi)}\subseteq \{\varepsilon|\varphi_\varepsilon\in
(N_a(x))^c\}$. It follows from Definition \ref{nax} that
$|x(\varphi_\varepsilon)|\ge
\dot{\alpha}_a(\varphi_\varepsilon)=(i(\varphi))^a\varepsilon^a,
~\forall~ 0<\varepsilon<\tau(\varphi)$ and  therefore
$|\frac{1}{x(\varphi_\varepsilon)}|\le (i(\varphi))^{-a}\varepsilon^{-a},
~\forall~ 0<\varepsilon<\tau(\varphi)$. This implies that
$\frac{1}{x}\in\mathcal{E}_f^M(\Kset)$ and, since $x\frac{1}{x}=1$,
follows that $x$ is a unit, a contradiction.  On the other hand if
$\mathcal{X}_{(N_a(x))^c}\in\mathcal{N}_f(\Kset)$ then,  since
$\mathcal{X}_{(N_a(x))^c}=1-\mathcal{X}_{N_a(x)}$, it  follows that
$\mathcal{X}_{N_a(x)}=1, ~\forall~a\in\Nset$ and hence, from Lemma 
\ref{vert} $(i)$, it follows that $x\in\mathcal{N}_f(\Kset)$, a 
contradiction. It follows that there must exist an  $\ a\in\Nset$ 
such that $N_a(x)\in\mathcal{S}_f$. 

The last affirmation follows immediately of the definition of $S$.
\end{proof}

The follow result is fundamental in proving many other results. 
Sometimes we shall refer to it as the Approximation Theorem.
\begin{The}\label{aproxim} {\tmsamp{Let
  $x\in\overline{\Kset}_f, ~(x\ne 0)$ be a non-unit. 
Then  only one of the following conditions holds:
\begin{enumerate}
\item[$(a)$] There exists $S\in\mathcal{S}_f$ and $a\in\Nset$ such that
\begin{enumerate}    
\item[$(i)$] $x\mathcal{X}_S=0$;
\item[$(ii)$] $|x\mathcal{X}_{S^c}|\ge \dot{\alpha}_a\mathcal{X}_{S^c}$ 
(i.e., there exists $\hat{x}$ representative of $x$ such that
    $|\hat{x}(\varphi)|\ge\hat{\dot{\alpha}}_a(\varphi),
    ~\forall~\varphi\in\mathcal{A}_0(\Kset))$.
\end{enumerate}  
\item[$(b)$] There exist sequences $\{a_n\}_{n\in\Nset}\subset\Nset$ and
  $\{S_n\}_{n\in\Nset}\subset \mathcal{S}_f$ such that
\begin{enumerate}
\item[$(i)$] $S_n\supset S_{n+1}, ~a_n<a_{n+1}$ 
and $\underset{n\to\infty}{\lim} a_n=\infty$;
\item[$(ii)$]
  $x\mathcal{X}_{S_n}\underset{n\rightarrow\infty}{\longrightarrow} 0$;
\item[$(iii)$] $|x\mathcal{X}_{S_n}|<\dot{\alpha}_{a_n}$.
\end{enumerate}    
 \end{enumerate}}}
\end{The}

\begin{proof} Suppose that $x$ does not satisfy $(a)$. We will show
  that condition $(b)$ holds. As $x\in\overline{\Kset}_f\setminus \{0\}$ and 
$x\notin\Inv(\overline{\Kset}_f)$ it follows, from Proposition \ref{creio}, 
that exists $a_1\in\Nset$ such that if $S_1:=N_{a_1}(x)$, then
$S_1\in\mathcal{S}_f$ and $|x\mathcal{X}_{S_1}|<\dot{\alpha}_{a_1}$.  
Let $x_2=x\mathcal{X}_{S_1}$. If $x_2=0$,
then $x\mathcal{X}_{S_1}=0$ and so $(a)$ $(i)$ holds. From Definition 
\ref{nax} we also have that
$|x\mathcal{X}_{S_1^c}|\ge \dot{\alpha}_{a_1}\mathcal{X}_{S_1^c}$,
i.e.,  $(a)$ $(ii)$ holds and hence $x$ it satisfies $(a)$, contrary
to our assumption. So we have that $x_2\neq 0$.

$x_2\mathcal{X}_{S_1^c}=0$ with $S_1\in\mathcal{S}_f$ and hence $x_2$
it is a non-trivial zero divisor and so is a non-unit. This allows us 
to proceed by induction. For completeness we show how to accomplish
the inductive step: There exists $a_{n+1}\in\Nset$ 
such that if $\tilde{S}_{n+1}:=N_{a_{n+1}}(x_{n+1})$, then
$\tilde{S}_{n+1}\in\mathcal{S}_f$ and
$|x_{n+1}\mathcal{X}_{\tilde{S}_{n+1}}|<\dot{\alpha}_{a_{n+1}}\Rightarrow
|x\mathcal{X}_{S_n\cap \tilde{S}_{n+1}}|<\dot{\alpha}_{a_{n+1}}$. Let
$S_{n+1}:=S_n\cap \tilde{S}_{n+1}$. Then $S_{n+1}\subset
S_n\in\mathcal{S}_f$ and so $ S_{n+1}\in \mathcal{S}_f$ and
$|x\mathcal{X}_{S_{n+1}}|<\dot{\alpha}_{a_{n+1}}$. Definition
\ref{nax} tells us that 
\begin{equation}\label{eq:compara}
\dot{\alpha}_{a_{n+1}}\le |x_{n+1}\mathcal{X}_{\tilde{S}_{n+1}^c}|=
|x\mathcal{X}_{S_n\cap\tilde{S}_{n+1}^c}|=|x\mathcal{X}_{\tilde{S}_{n+1}^c}|,    
\end{equation}
therefore $\tilde{S}_{n+1}^c\subset S_n$. In fact, if
$\varphi\in\tilde{S}_{n+1}^c$, then $\dot{\alpha}_{a_{n+1}}(\varphi)\le
|x_{n+1}(\varphi)|=|x(\varphi)\mathcal{X}_{S_n}(\varphi)|$ and since 
$\dot{\alpha}_{a_{n+1}}=(i(\varphi))^{a_{n+1}}>0$ we have that
$|x(\varphi)\mathcal{X}_{S_n}(\varphi|>0\Rightarrow
\mathcal{X}_{S_n}(\varphi)\ne 0$, i.e.,
$\mathcal{X}_{S_n}(\varphi)=1\Rightarrow \varphi\in S_n$ hence 
$\tilde{S}_{n+1}^c\subset S_n$ and $S_n\cap
\tilde{S}_{n+1}^c=\tilde{S}_{n+1}^c$. Now 
$\ \forall \varphi\in\tilde{S}_{n+1}^c\subset S_n$ we have, from 
\ref{eq:compara}, that
$$\dot{\alpha}_{a_{n+1}}(\varphi)\le
|x_{n+1}(\varphi)|=|x(\varphi)|<\dot{\alpha}_{a_n}(\varphi).$$ From there, 
we have that $\dot{\alpha}_{a_{n+1}}(\varphi)<\dot{\alpha}_{a_n}(\varphi),
~\forall~ \varphi\in\tilde{S}_{n+1}^c$ which implies that
$$(i(\varphi))^{a_{n+1}}<(i(\varphi))^{a_n}, ~\forall~
\varphi\in\tilde{S}_{n+1}^c.$$ Hence $a_n<a_{n+1}$ for some
$\varphi\in\tilde{S}_{n+1}^c$ such that $0<i(\varphi)<1$. Such
$\varphi\in\tilde{S}_{n+1}^c$ it always exists, therefore
$\tilde{S}_{n+1}^c\in\mathcal{S}_f$ \footnote{In fact, it takes
$\varphi\in\tilde{S}_{n+1}^c$, then $i(\varphi_\varepsilon)=\varepsilon
i(\varphi)$ and
$\{\varepsilon|\varphi_\varepsilon\in\tilde{S}_{n+1}^c\}\in\mathcal{S}$,
therefore exists $\varepsilon_0$ such that 
$\varphi_{\varepsilon_0}\in\tilde{S}_{n+1}^c$ and
$i(\varphi_{\varepsilon_0})<1$, i.e., $\varepsilon_0 i(\varphi)<1$,
pois $i(\varphi_\varepsilon)=\varepsilon i(\varphi)\to 0$ when
$\varepsilon\to 0$.}. 

In this way we construct sequences $\{S_n\}_{n\in\Nset}$ and 
$\{a_n\}_{n\in\Nset}$ satisfying conditions $(i)$ and $(iii)$. 
We now show that condition $(ii)$ also holds: from $(iii)$ it
follows that for each $n\in\Nset$, 
$|(x\mathcal{X}_{S_n})(\varphi)|<\dot{\alpha}_{a_n}(\varphi),
~\forall~\varphi\in\mathcal{A}_0(\Kset)$ and hence $ \|x\mathcal{X}_{S_n}\|\le
\|\dot{\alpha}_{a_n}\|=e^{-a_n}\underset{n\rightarrow\infty}{\longrightarrow}
0$.
\end{proof}

%%%%%%Parei aqui!!!!%%%%%%%%%
\begin{The}\label{mosca} {\tmsamp{$x\in\Inv(\overline{\Kset}_f)$ if
      and only if there are
  $r>0, ~\tau:\mathcal{A}_0(\Kset)\to ]0,1]$ such that
  $|\hat{x}(\varphi_\varepsilon)|\ge\dot{\alpha}_r(\varphi_\varepsilon), 
~\forall~0<\varepsilon<\tau(\varphi),$ where $\hat{x}$ is 
a representative of $x$.}}
\end{The}

\begin{proof} Suppose that $x\in\Inv(\overline{\Kset})$. Then, 
from Lemma \ref{vert},  we have that 
$\mathcal{X}_{N_q(\hat{x})}=0, ~\forall~ q\in\Nset$, i.e.,
$\hat{\mathcal{X}}_{N_q(\hat{x})}\in\mathcal{N}_f$. Lemma 
\ref{antes} implies  that there 
exists $\tau:\mathcal{A}_0(\Kset)\to ]0,1]$ such that
$\ \forall \varphi\in\mathcal{A}_q(\Kset)$ we have
$I_{\tau(\varphi)}\subset\{\varepsilon|\varphi_\varepsilon\in
(N_q(\hat{x}))^c\}$. Thus, from Definition \ref{nax}, it follows that 
$|\hat{x}(\varphi_\varepsilon)|\ge \dot{\alpha}_q(\varphi_\varepsilon),
~\forall~0<\varepsilon<\tau(\varphi),$ where $\hat{x}$ is a
representative of $x$. Hence we may take $q=r$.

Conversely, if there exists $r>0, ~\tau:\mathcal{A}_0(\Kset)\to
]0,1]$ such that
$$|\hat{x}(\varphi_\varepsilon)|\ge\dot{\alpha}_r(\varphi_\varepsilon),
~\forall~ 0<\varepsilon<\tau(\varphi),$$ then
$$|\hat{x}(\varphi_\varepsilon)|\ge (i(\varphi))^r\varepsilon^r,
~\forall~ 0<\varepsilon<\tau(\varphi)$$ and hence
$$\left|\frac{1}{\hat{x}(\varphi_\varepsilon)}\right|\le
  (i(\varphi))^{-r}\varepsilon^{-r}, ~\forall~ 0<\varepsilon<\tau(\varphi).$$ 
  
We assert that $1/\hat{x}\in\mathcal{E}_f^M(\Kset)$:  taking 
$p=r>0, ~C=C(\varphi)=(i(\varphi))^{-p}>0$ and
$\eta=\eta(\varphi)=\tau(\varphi)>0$ to get 
$$\left|\frac{1}{\hat{x}(\varphi_\varepsilon)}\right|
\le(i(\varphi))^{-r}\varepsilon^{-r}=C\varepsilon^{-p},
~\forall~ 0<\varepsilon<\tau(\varphi).$$ Therefore  
$x\in\Inv(\overline{\Kset}_f)$.
\end{proof}

\begin{Lem}\label{rep} {\tmsamp{Let $x\in\overline{\Kset}_f, ~(x\ne 0)$ be 
a non-unit. Then there exists $a\in\Nset$ such that
  $y:=x(1-\mathcal{X}_{N_a(\hat{x})})+\mathcal{X}_{N_a(\hat{x})}
\in\Inv(\overline{\Kset}_f)$.}}
\end{Lem}

\begin{proof} By Proposition \ref{creio}, there exists $a\in\Nset$
such that $N_a(\hat{x})\in\mathcal{S}_f$ hence
  $\mathcal{X}_{N_a(\hat{x})}\notin \{0, 1\}$. Define 
$\hat{y}:=\hat{x}(1-\hat{\mathcal{X}}_{N_a(\hat{x})})+
\hat{\mathcal{X}}_{N_a(\hat{x})}.$ Then we see easily that the class 
of $\hat{y}$ defines a unit.
\end{proof}

\begin{The}\label{rad} {\tmsamp{Let $x\in\overline{\Kset}_f, ~(x\ne 0)$ 
be non-unit. Then there exists a maximal ideal $\mathfrak{m}$ of
$\overline{\Kset}_f$ such that $x\notin\mathfrak{m}$. Hence the
Jacobson radical $\Rad(\overline{\Kset}_f)=\{0\}$.}}
\end{The}

\begin{proof} By Lemma \ref{rep}, there exists $a\in\Nset$ such that
$y=x(1-\mathcal{X}_{N_a(\hat{x})})+\mathcal{X}_{N_a(\hat{x})}\in
\Inv(\overline{\Kset}_f).$ As $\mathcal{X}_{N_a(\hat{x})}\notin
\{0,1\}$, 
$\mathcal{X}_{N_a(\hat{x})}\notin \Inv(\overline{\Kset}_f)$ and 
$\overline{\Kset}_f$ is a ring with unit, there exists a maximal ideal 
$\mathfrak{m}\triangleleft \overline{\Kset}_f$ such that
$\mathcal{X}_{N_a(\hat{x})}\in \mathfrak{m}$. We shall prove that 
$x\notin \mathfrak{m}$: if $x\in \mathfrak{m}$, then
$x\mathcal{X}_{(N_a(\hat{x}))^c}=x(1-\mathcal{X}_{N_a(\hat{x})})\in\mathfrak{m}$
hence
$y=x(1-\mathcal{X}_{N_a(\hat{x})})+
\mathcal{X}_{N_a(\hat{x})}\in\mathfrak{m}\bigcap
\Inv(\overline{\Kset}_f)$, 
a contradiction. 
\end{proof}

We can now state the main result of this section which completely
describes the maximal ideals of $\overline{\Kset}_f$ and shows that
the unit group is dense and open. 

\begin{The}\label{impor}{\tmsamp{
\begin{enumerate}
\item[$(1)$] Let $\mathfrak{m}\subset \overline{\Kset}_f$ be an ideal. 
Then $\mathfrak{m}$ is maximal if and only if 
$\mathfrak{m}=\overline{g_f(\mathcal{F}_\mathfrak{m})}$.
\item[$(2)$] $\Inv(\overline{\Kset}_f)$ is an open and 
dense subset of $\overline{\Kset}_f$.
\end{enumerate}}}
\end{The}

\begin{proof} 

%The proof is the same as in \cite{AJ} and is therefore 
%omitted. 
%\end{proof}

$(1)$ $(\Rightarrow)$ Suppose that $\mathfrak{m}$ is a maximal ideal of
$\overline{\Kset}_f$. Then $\mathfrak{m}$ is a prime ideal
of $\overline{\Kset}_f$ and from the Theorem of the approach there
exists only $\mathcal{F}=\mathcal{F}_\mathfrak{m}$ such that
$g_f(\mathcal{F})\subset \mathfrak{m}$. Let
$x\in\mathfrak{m}\setminus g_f(\mathcal{F})$. We now construct a
sequence $\{x_n\}_{n\in\Nset}$ in $g_f(\mathcal{F})$ such that
$x_n\underset{n\rightarrow\infty}{\longrightarrow} x$ in
$\overline{\Kset}_f$. For this we show that
$x$ satisfies the condition $(b)$ of the Theorem of the
Approach. Indeed, if $x$ does not satisfy condition
$(b)$ of the Theorem of the approach then, has been that $x$
it satisfies the condition $(a)$ of the Theorem of the approach, i.e.,
there exists
$S\in\mathcal{S}_f$ and $a\in\Nset$ such that
\begin{enumerate}
\item[$(i)$] $x\mathcal{X}_S=0$;
\item[$(ii)$] $|x\mathcal{X}_{S^c}|\ge \dot{\alpha}_a\mathcal{X}_{S^c}$.
\end{enumerate}
One remembers that this $S$ is accurately $N_a(x)$. Suppose that
$S\in\mathcal{F}$. Then $$\mathcal{X}_S\in g_f(\mathcal{F})\subset
\mathfrak{m}\Rightarrow \mathcal{X}_S\in\mathfrak{m}.$$ Thus,
$$y:=x(1-\mathcal{X}_S)+\mathcal{X}_S\in\mathfrak{m}$$ which is
nonsense. Therefore from Lemma \ref{rep} $y$ is invertible. Now,
assume that $S^c\in\mathcal{F}$. Then
$$\mathcal{X}_{S^c}=(1-\mathcal{X}_S)\in g_f(\mathcal{F})$$ and from
$(a)$ part $(i) ~x\mathcal{X}_S=0$ of Theorem of the approach it follows that
$$x\mathcal{X}_{S^c}=x(1-\mathcal{X}_S)=x\in g_f(\mathcal{F})$$ which
is nonsense. Therefore, hypothecally, $x\notin g_f(\mathcal{F})$. Hence,
we conclude that $x$ cannot satisfy condition $(a)$ of the Theorem 
of the Approach and, therefore, $x$ satisfies condition $(b)$ of the 
Theorem of the Approach, i.e., there exists sequences
$\{a_n\}\subset\Nset$ and $\{S_n\}\subset \mathcal{S}_f$ such that    
\begin{enumerate}
\item[$(i)$] $S_n\supset S_{n+1}, ~a_n<a_{n+1}$ and 
$$\lim_{n\to \infty}a_n=\infty;$$
\item[$(ii)$] $x\mathcal{X}_{S_n}\underset{n\rightarrow
    \infty}{\longrightarrow} 0$;
\item[$(iii)$] $|x\mathcal{X}_{S_n}|<\dot{\alpha}_{a_n}$
\end{enumerate}    
We affirm that $S_n^c\in\mathcal{F}$. Indeed, if this will not be the case,
then $S_n\in\mathcal{F}$ end therefore, $$\mathcal{X}_{S_n}\in
g_f(\mathcal{F})\subset\mathfrak{m}\Rightarrow
\mathcal{X}_{S_n}\in\mathfrak{m}$$ and once again we would have that
$$y=x(1-\mathcal{X}_{S_n})+\mathcal{X}_{S_n}\in\mathfrak{m}$$ which is
nonsense. Therefore from Lemma \ref{rep} $y\in\Inv(\overline{\Kset}_f)$.
Thus, $S_n^c\in\mathcal{F}$ which implies $\mathcal{X}_{S_n^c}\in
g_f(\mathcal{F})$. However $\mathcal{X}_{S_n^c}=1-\mathcal{X}_{S_n}$, hence
$$x_n=x\mathcal{X}_{S_n^c}=x-x\mathcal{X}_{S_n}$$ is a sequence in
$g_f(\mathcal{F})$ such that
$x_n\underset{n\rightarrow\infty}{\longrightarrow} x$, since from 
$(ii)$ of $(b)$ of the Theorem of the approach we have that 
$x\mathcal{X}_{S_n}\underset{n\rightarrow \infty}{\longrightarrow} 0$. 
Hence, $\mathfrak{m}=\overline{g_f(\mathcal{F}_\mathfrak{m})}$.
$(\Rightarrow)$ From Lemma \ref{idpro} $g_f(\mathcal{F})$ is a proper ideal 
of $\overline{\Kset}_f$ and, therefore, $g_f(\mathcal{F})\subset
\mathfrak{m}$ for some maximal ideal $\mathfrak{m}$ of
$\overline{\Kset}_f$. More still, for the Theorem of the approach 
we have that $\mathcal{F}=\mathcal{F}_\mathfrak{m}$ and the conclusion 
comes immediately from necessary condition.

$(2)$ We know of the Corollary \ref{aberto} that
$\Inv(\overline{\Kset}_f)$ it is an open subset of
$\overline{\Kset}_f$, hence, it remains to show that
$\Inv(\overline{\Kset}_f)$ is a subset dense of
$\overline{\Kset}_f$. That is, all element $x\in\overline{\Kset}_f$ 
is limit of some sequence of elements of
$\Inv(\overline{\Kset}_f)$. For the Lemma \ref{feixo} $(i)$, we have
that $0\in\overline{\Kset}_f$ is limit of a sequence in
$\Inv(\overline{\Kset}_f)$ \footnote{See the sequence 
$\{\dot{\alpha}_n\}_{n\in\Nset}$ in $\Inv(\overline{\Kset}_f)$ that is
$\tau_{sf}$-convergent the zero.}. It is enough then to study the case
where $x\in\overline{\Kset}_f\setminus \{0\}$ 
and $x\notin\Inv(\overline{\Kset}_f)$, therefore if
$x\in\Inv(\overline{\Kset}_f)$ we take the sequence constant. 
For the Theorem of the  da Approach $x$ it satisfies $(a)$ or $x$
it satisfies $(b)$. In each one of the cases we go to construct
sequences in $\Inv(\overline{\Kset}_f)$ that they are
$\tau_{sf}$-convergent the $x$. Initially let us assume that $x$ 
it satisfies $(a)$ and defines
$$x_n:=x(1-\mathcal{X}_S)+\dot{\alpha}_n\mathcal{X}_S.$$ We go to show that
$x_n\in\Inv(\overline{\Kset}_f), ~\forall~ n\in\Nset$ and that
$x_n\underset{n\rightarrow \infty}{\longrightarrow} x$. 
Indeed, for the item $(i)$ of $(a)$, we have that
$$x_n=x+\dot{\alpha}_n\mathcal{X}_S, ~\mbox{pois}~ x\mathcal{X}_S=0.$$ 
Now remeber that $S=N_a(x)$ and
$$x_n(\varphi)=\left\{\begin{array}{ll}
x(\varphi)+\dot{\alpha}_n(\varphi), &\mbox{se}~\varphi\in S\\
x(\varphi), &\mbox{se}~\varphi\in S^c,
\end{array}\right.
$$
we have that $$|x_n(\varphi)|\ge \dot{\alpha}_n(\varphi), ~\forall~\varphi
~\mbox{e}~\forall~ n\in\Nset.$$ For the Theorem \ref{mosca},
$x_n\in\Inv(\overline{\Kset}_f), ~\forall~n\in\Nset$, i.e.,
$\{x_n\}_{n\in\Nset}$ is a sequence in
$\Inv(\overline{\Kset}_f)$. Now, as 
$$x_n=x+\dot{\alpha}_n\mathcal{X}_S$$ follows that
$$x_n-x=\dot{\alpha}_n\mathcal{X}_S.$$ Hence for the pelo Corollary \ref{norma}
$(iv)$, we have that
$$\|x_n-x\|=\|\dot{\alpha}_n\mathcal{X}_S\|=e^{-n}\|\mathcal{X}_S\|=e^{-n},$$
i.e., $$\|x_n-x\|=e^{-n}$$ and as
$e^{-n}\underset{n\rightarrow\infty}{\longrightarrow} 0$ follows that
$\|x_n-x\|\underset{n\rightarrow\infty}{\longrightarrow} 0$ what it
implies $x_n\underset{n\rightarrow\infty}{\longrightarrow} x$. 
Finally, let us assume that $x$ satisfies the condition $(b)$ and defines
$$x_n:=x(1-\mathcal{X}_{S_n})+\dot{\alpha}_{a_n}\mathcal{X}_{S_n}.$$
We go to show that $x_n\in\Inv(\overline{\Kset}_f), ~\forall~n\in\Nset$
and that $x_n\underset{n\rightarrow\infty}{\longrightarrow} x$. Indeed, 
$$x_n(\varphi)=\left\{\begin{array}{ll}
\dot{\alpha}_{a_n}(\varphi), &\mbox{se}~\varphi\in S_n\\
x(\varphi), &\mbox{se}~\varphi\in S_n^c.
\end{array}\right.
$$
Recalling that in the demonstration of the Theorem of the approach
$S_n=N_{a_n}(x)$, we have that
$$|x_n(\varphi)|\ge\dot{\alpha}_{a_n}(\varphi), ~\forall~\varphi,
~\forall~n\in\Nset.$$ And, one more time, for the Theorem \ref{mosca}, come
$x_n\in\Inv(\overline{\Kset}_f), ~\forall~n\in\Nset$. Since
$$x_n=x-x\mathcal{X}_{S_n}+\dot{\alpha}_{a_n}\mathcal{X}_{S_n}$$
follows that 
\begin{eqnarray*}
\|x_n-x\|&=&\|\dot{\alpha}_{a_n}\mathcal{X}_{S_n}-x\mathcal{X}_{S_n}\|\\
&\le&\|\dot{\alpha}_{a_n}\mathcal{X}_{S_n}+x\mathcal{X}_{S_n}\|
~(\mbox{for the
  Corollary \ref{norma} $(i)$, we have that})\\  
&\le&\max\{\|\dot{\alpha}_{a_n}\mathcal{X}_{S_n}\|,\|x\mathcal{X}_{S_n}\|\}
~(\mbox{for the Corollary \ref{norma} $(iv)$, we have that})\\
&=&\max\{e^{-a_n},\|x\mathcal{X}_{S_n}\|\}.
\end{eqnarray*}
But for itens $(i)$ and $(ii)$ of $(b)$, we have that
$\max\{e^{-a_n},\|x\mathcal{X}_{S_n}\|\}\underset{n\rightarrow\infty}{\longrightarrow}0\Rightarrow\|x_n-x\|\underset{n\rightarrow\infty}{\longrightarrow}
0$ hence $x_n\underset{n\rightarrow\infty}{\longrightarrow}
x$. This sample our assertive one.
\end{proof}

\section{Algebraic Properties of $\overline{\Kset}_f$}\label{sec-04}
In this section we introduce a partial order in $\overline{\Rset}_f$
the ring of the Colombeau's full generalized numbers (on $\Rset$)
which we shall prove induces a total order in every residualy class
field. The base for this section is the works of Aragona, Juriaans, 
Oliveira and Scarpalézos \cite{AJOS} and Aragona, Fernandez and
Juriaans \cite{NTCA} who developed research, mostly related to
the ring of Colombeau's simplified generalized numbers    
$\overline{\Kset}$ (where $\Kset$ is $\Rset$ or $\Cset$). Actually we
prove a stronger result which will allow us to prove, in the next
sub-section, that $\overline{\Kset}_f$ contains minimal prime
ideals. Before we go on we define some conventions. 
\begin{Not}
{\tmsamp{If
$\mathcal{F}\in\mathit{P}_*(\mathcal{S}_f)$ then
${g_f}_r(\mathcal{F})$ denotes the ideal of $\overline{\Rset}_f$
generated by the characteristic functions of elements of $\mathcal{F}$
and $g_f(\mathcal{F})$ the ideal of $\overline{\Cset}_f$ generated by
the same function}} (see Section \ref{sec-03}). 
\end{Not}
\subsection{Relation of the order on $\overline{\Rset}_f$}\label{subsec-041}
The Lemma below from Aragona, Fernandez and Juriaans
\cite{NTCA} is the base for the order definition.
\begin{Lem}\label{base} {\tmsamp{For all $x\in\overline{\Rset}_f$ the following
statements are equivalent: 
\begin{enumerate}
\item[$(i)$] Every representative $\hat{x}$ of $x$ satisfies the condition\\
$$
(*)\left|\begin{array}{ll} 
\exists~ N\in\Nset ~\mbox{tal que}~ \forall~ b>0
~\forall~\varphi\in\mathcal{A}_N(\Kset) ~\exists \\
\eta=\eta(b,\varphi)\in I ~\mbox{tal que}~
\hat{x}(\varphi_\varepsilon)\ge -\varepsilon^b, 
~\forall~ \varepsilon\in I_\eta.
\end{array}\right.
$$  
\item[$(ii)$] There exists a representative $\hat{x}$ of $x$ such that
$\hat{x}$ satisfying $(*)$.
\item[$(iii)$] There exists a representative $x_*$ of $x$ such that
$x_*(\varphi)\ge 0, ~\forall~\varphi\in\mathcal{A}_0(\Kset)$.
\item[$(iv)$] There exists $N\in\Nset$ and a representative $x_*$ of
$x$ such that $x_*(\varphi)\ge 0, ~\forall~\varphi\in\mathcal{A}_N(\Kset)$. 
\end{enumerate}}}
\end{Lem}
\begin{proof} See \cite{NTCA}.
\end{proof}

\begin{Def}\label{base-1} {\tmsamp{An element $x\in\overline{\Rset}_f$ is said
to be non-negative, quasi-positive or q-positive, if it has a
representative satisfying one of the conditions of the Lemma
\ref{base}. We shall denote this by $x\ge 0$. We shall also define $x$
to be non-positive, quasi-negative or q-negative if
$-x$ is q-positive and we denote this by $x\le 0$. If
$y\in\overline{\Rset}_f$ is another element then we write $x\le y$ 
(resp. $x\ge y$) if $y-x$ (resp. $x-y$) is q-positive.}}
\end{Def}

\begin{Rem} {\tmsamp{This definition is not a total order in
  $\overline{\Rset}_f$. To see this note that $$\hat{x}(\varphi)=
\hat{\dot{\alpha}}_1(\varphi)\sin(\hat{\dot{\alpha}}_{-1}(\varphi)),
~\forall ~\varphi\in\mathcal{A}_0(\Kset)$$ gives to an element which
is neither q-positive nor q-negative. It does however define a
partial order such that the sum and product of q-positive elements
are q-positive.}}
\end{Rem}

For every $x\in\overline{\Kset}_f$, if $\hat{x}$ is a representative
of $x$, the function $|\hat{x}|:\mathcal{A}_0\to \Rset_+$ defined by
$|\hat{x}|(\varphi)=|\hat{x}(\varphi)|$ is of course a moderate
function and $|x|:=\cl[|\hat{x}|]$ is independent of the
representative $\hat{x}$ of $x$. This Colombeau's full generalized
number is called the absolute value of $x$. Hence we have the natural
aplication $$x\in\overline{\Kset}_f\mapsto |x|\in\overline{\Rset}_{f+}.$$

\begin{Def} {\tmsamp{Let $x\in\overline{\Rset}_f$. Then
  $x^+:=\frac{x+|x|}{2}$ and $x^-:=\frac{x-|x|}{2}$ are respectively
  called the q-positive and q-negative parts of $x$.}}   
\end{Def}

Note that $x^+$ and $x^-$ depend only on $x$.
\begin{Def} {\tmsamp{For a given $u\in\mathcal{E}_f^M(\Kset)$ we define the
  next functions $\theta_u:\mathcal{A}_0(\Kset)\to \Kset$ such that
  $\theta_u(\varphi)=\exp(-i\Arg(u(\varphi)))$ and
  $\theta_u^{-1}:\mathcal{A}_0(\Kset):\to \Kset$ such that 
$\theta_u^{-1}(\varphi)=\exp(i\Arg(u(\varphi)))$,
where $\Arg(u(\varphi))$ denote the argument
of $u(\varphi)\in\Kset$, with the convention that $\Arg(0):=0$.}}
\end{Def}

In the case $\Kset=\Rset$ the images of $\theta_u$ and $\theta_u^{-1}$
are contained in $\{-1,1\}$. It is also clear that these are moderate
functions and inverses of each other. Moreover, we have
$u(\varphi)\theta_u(\varphi)=|u(\varphi)|,
~\forall~\varphi\in\mathcal{A}_0(\Kset)$. Therefore, if we denote for
$|u|(\varphi):=|u(\varphi)|, ~\forall~\varphi\in\mathcal{A}_0(\Kset)$,
we can to write $u\theta_u=|u|$ and, therefore, $|u|\in\mathcal{E}_f^M(\Rset)$.
\begin{Def} {\tmsamp{For a given $u\in\mathcal{E}_f^M(\Kset)$ we defined
  $\Theta_u=\cl[\theta_u]$ and $\Theta_u^{-1}=\cl[\theta_u^{-1}]$}}
\end{Def}

Since $\Theta_u\Theta_u^{-1}=1=\Theta_u^{-1}\Theta_u$ we have that
they are units. Note however that these functions depend on the
representative. The following proposition is easily proved.
\begin{Pro}\label{facil} {\tmsamp{Let $x,y\in\overline{\Rset}_f$. Then:
\begin{enumerate}
\item[$(i)$] $x=x^+$ iff $x=|x|$ iff $x$
    is q-positive.
\item[$(ii)$] $x=x^-$ iff $x=-|x|$ iff $x$
    is q-negative.
\item[$(iii)$] $(-x)^+=-(x^-)$ e $(-x)^-=-(x^+)$.
\item[$(iv)$] $|x|\ge 0\le x^+, ~x^-\le 0\ge -x^+, ~|-x|=|x|$ e
  $|x|\ge x$.
\item[$(v)$] $|x+y|\le |x|+|y|, ~||x|-|y||\le |x-y|$ (triangular inequality).
\item[$(vi)$] If $x\le y$ and $-x\le y$, then $|x|\le y$.
\item[$(vii)$] $x^+=x\left(\frac{1+\Theta_{\hat{x}}}{2}\right)$ and
  $x^-=x\left(\frac{1-\Theta_{\hat{x}}}{2}\right)$.
\item[$(viii)$] If
  $A=\{\varphi\in\mathcal{A}_0(\Kset)|\hat{x}(\varphi)\ge 0\}$ then
  $x^+=x\mathcal{X}_A$ and $x^-=x\mathcal{X}_{A^c}$.
\end{enumerate}}}
\end{Pro}
%\begin{proof}

%\end{proof}
\begin{Rem} {\tmsamp{Note that if $z\in\overline{\Cset}_f$ then
  $|z|\in\overline{\Rset}_{f}, ~|z|\ge 0$ and so we may apply
  Proposition \ref{facil} where possible. In particular, the
triangular inequalities hold in this context.}}
\end{Rem}

\begin{Pro}[Convexity of ideals]\label{convex} {\tmsamp{Let
  $\mathfrak{J}$ be an ideal of $\overline{\Kset}_f$ and
  $x,y\in\overline{\Kset}_f$. Then:
\begin{enumerate}
\item[$(i)$] $x\in\mathfrak{J}$ if, and only if,
    $|x|\in\mathfrak{J}$.
\item[$(ii)$] If $x\in\mathfrak{J}$ and $|y|\le |x|$ then $y\in\mathfrak{J}$.
\item[$(iii)$] If $\Kset=\Rset$ and $0\le y\le x$ then $y\in\mathfrak{J}$.
\end{enumerate}}}
\end{Pro}

\begin{proof} $(i)$ If $x\in\mathfrak{J}$, then
$|x|=x\Theta_{\hat{x}}\in\mathfrak{J}$. Therefore $\mathfrak{J}$ is an
ideal. Reciprocally, if $|x|\in\mathfrak{J}$, then
$|x|=x\Theta_{\hat{x}}\in\mathfrak{J}$. Now, since
$\Theta_{\hat{x}}\in\Inv(\overline{\Kset}_f)$ follows that
$x=|x|\Theta_{\hat{x}}^{-1}\in\mathfrak{J}$, i.e., $x\in\mathfrak{J}$.

$(ii)$ If $|y|=|x|$, then from $(i), ~|x|\in\mathfrak{J}$. Therefore
$x\in\mathfrak{J}$ hence
$|y|\in\mathfrak{J}$ and from the reciprocal of $(i)$ it follows that
$y\in\mathfrak{J}$. Now, if $|y|<|x|$, then $|x|\ne 0$. Therefore if
$|x|=0$, then $|y|<0$ which is nonsense. Therefore $|y|\ge 0$. Let
$u=\frac{|y|}{|x|}$. Then $|y|=u|x|\in\mathfrak{J}$. Therefore from $(i)$
$|x|\in\mathfrak{J}$ ($x\in\mathfrak{J}$) and $\mathfrak{J}$ is a
ideal of $\overline{\Kset}_f$. Once again, from the reciprocal one of $(i)$, 
we have that $y\in\mathfrak{J}$. 

$(iii)$ Follows from $(ii)$ and Proposition \ref{facil}.
\end{proof}

\begin{Rem} {\tmsamp{If $z\in\overline{\Cset}_f$ then it is clear that we may
  write $z=x+iy$, with $x,y\in\overline{\Rset}_f$, where $i$
  is the class of the constant function $\sqrt{-1}$. We define $\Re e(z):=x$ and
  $\Im m(z):=y$, the real and imaginary part of $z$, respectively. 
Clearly, if $\hat{z}=\hat{x}+\sqrt{-1}\hat{y}$ is a
  representative of $z$ then $\hat{x}$ and $\hat{y}$ are
  representatives of $x$ and $y$, respectively. It is also clear that
$\overline{\Cset}$ is a $\overline{\Rset}$-module and that the maps 
$\Re e,\Im m:\overline{\Cset}_f\to\overline{\Rset}_f$ are
$\overline{\Rset}_f$-epimorphisms. Hence if
$\mathfrak{J}\lhd\overline{\Cset}_f$ is an ideal then its image by
these epimorphisms are ideals ideals of $\overline{\Rset}_f$ and it is
easily seen that they coincide; it will be denoted by
$\mathfrak{J}_r$ and called the real part of
  $\mathfrak{J}$. We have by Proposition \ref{convex} that
$\mathfrak{J}_r\subset\mathfrak{J}$.}}
\end{Rem}

Note that the involutions $c:\Cset\to\Cset$ defines by
$c(z)=\overline{z}$ extends to an involution
$\overline{\Cset}_f\to \overline{\Cset}_f$ of $\overline{\Cset}_f$. 
We shall call this involution conjugation. The following result is
clear.
\begin{Lem} {\tmsamp{Let $\mathfrak{J}\lhd\overline{\Cset}_f$ be an ideal of
  $\overline{\Cset}_f$. Then:
  \begin{enumerate}
  \item[$(i)$] $\mathfrak{J}_r\subset\mathfrak{J}$.
\item[$(ii)$] $\mathfrak{J}=\mathfrak{J}_r+i\mathfrak{J}_r$ and
  $\mathfrak{J}$ is invariant under conjugation.
\item[$(iii)$] $\mathfrak{J}_r=\mathfrak{J}\cap\overline{\Rset}_f$.  
\end{enumerate}}}
\end{Lem}
%\begin{proof} 

% %\end{proof}
\begin{Cor} {\tmsamp{Let $\mathcal{F}\in\mathit{P}_*(\mathcal{S}_f)$. Then: 
  \begin{enumerate}
  \item[$(i)$] ${g_f}_r(\mathcal{F})$ is the real part of
    $g_f(\mathcal{F})$ 
(see Definition \ref{degf}). 
\item[$(ii)$] If $z\in\overline{\Cset}_f$ then $z\in
  g_f(\mathcal{F})$ if and only if $|z|\in{g_f}_r(\mathcal{F})$.  
\end{enumerate}}}
\end{Cor}
% %\begin{proof}

% %\end{proof}
\begin{Lem}\label{val} {\tmsamp{Let 
$\mathcal{F}\in\mathit{P}_*(\mathcal{S}_f)$ and
$x,y\in\overline{\Rset}_f$. Then the following hold:
\begin{enumerate}
\item[$(i)$] If $(x-y)\in {g_f}_r(\mathcal{F})$ and
$x^-\in{g_f}_r(\mathcal{F})$ then $y^-\in {g_f}_r(\mathcal{F})$.
\item[$(ii)$] The part $x^+$ or $x^-$ of $x$ belongs to ${g_f}_r(\mathcal{F})$.
\end{enumerate}}}
\end{Lem}

\begin{proof} $(i)$ We have from Proposition \ref{facil} item $(v)$ that
\begin{eqnarray*}
  ||x^{-}-y^{-}||&=&\left|\left|\frac{x-|x|}{2}-\frac{y-|y|}{2}\right|\right|\\
&=&\frac{1}{2}|x-|x|-y+|y||\\
&\le&\frac{1}{2}(|x-y|+||y|-|x||)\\
&\le&|x-y|,
\end{eqnarray*}
i.e., $||x^{-}-y^{-}||\le |x-y|$. Since
$(x-y)\in{g_f}_r(\mathcal{F})$, it follows from Proposition
\ref{convex} item $(ii)$ that $|x^{-}-y^{-}|\in{g_f}_r(\mathcal{F})$
and from the reciprocal of \ref{convex} item $(i)$ we have that
$x^{-}-y^{-}\in{g_f}_r(\mathcal{F})$ and since
$x^{-}\in{g_f}_r(\mathcal{F})$ it follows that $y^{-}\in{g_f}_r(\mathcal{F})$.

$(ii)$ We can assume here that $x$ is q-positive and non 
q-negative hence, $x$ has a representative $\hat{x}$ such that
$\theta_{\hat{x}}\notin\{\pm 1\}$ what it means that if
$A:=\{\varphi\in\mathcal{A}_0(\Kset)|\theta_{\hat{x}}\equiv 1\}$,
then $A$ or $A^c$ belongs in $\mathcal{F}$ and the result follows from
Proposition \ref{facil} item $(viii)$.
\end{proof}

\begin{Def}\label{ordem} {\tmsamp{Let 
$\alpha\in\overline{\Rset}_f/{g_f}_r(\mathcal{F})$ be given. 
We say that $\alpha$ is non-negative, $\alpha\ge 0$, if $\alpha$ has a
representative $a\in\overline{\Rset}_f$ such that the part negative,
  $a^-$, belongs to $\in {g_f}_r(\mathcal{F})$.}} 
\end{Def}

Definition \ref{ordem} gives rise to ordering in the traditional
way. Lemma \ref{val} shows that Definition \ref{ordem} is
intrinsic, i.e., it does not depend on the representative. The
following lemma is easily shown and should be well known.
\begin{Lem}\label{hum} {\tmsamp{Let $(A,\le)$ be a commutative unitary
  partially ordered ring and $a,b\in A$. Then:
  \begin{enumerate}
  \item[$(i)$] $(A,\le)$ is totally ordered if and only if either $a\ge
    0$ or $-a\ge 0$.
\item[$(ii)$] If $\ge$ is a total order on $A$, the nil-radical of
  $A$, $\mathcal{N}(A)$, is zero, i.e., $\mathcal{N}(A)=0$ 
and if $a,b\ge 0\Rightarrow ab\ge 0$,
then $A$ is an integral domain.
  \end{enumerate}}}
\end{Lem}
%\begin{proof} 

%\end{proof}
We now come to our main result of this sub-section. 
\begin{The}\label{motor} {\tmsamp{Let
  $\mathcal{F}\in\mathit{P}_*(\mathcal{S}_f)$. Then
  $\left(\overline{\Rset}_f/{g_f}_r(\mathcal{F}),\le\right)$ is a
  totally ordered ring.}}
\end{The}
\begin{proof} The proof is similar that one presented in \cite{AJOS}
  (See Theorem $3.14$ pg-$8$).
\end{proof}

\subsection{Other algebraic properties of  $\overline{\Kset}_f$}\label{subsec-042}
We reserve this sub-section to give continuity to the algebraic facts
studied in the Section \ref{sec-03} where the maximal ideals of
$\overline{\Kset}_f$ were completely described. Now we completely
describe the minimal primes and show that $\overline{\Kset}_f$ is not
Von Neumann regular.

\begin{Not} {\tmsamp{If $A$ is a commutative unitary ring, we denote by
  $\mathcal{B}(A)$ the set of idempotents of $A$. Our first result
  describes completely the idempotents of $\overline{\Kset}_f$.}}
\end{Not}

\begin{The} \label{idemp} {\tmsamp{Let $e\in\overline{\Kset}_f$ be a non-trivial
  idempotente. Then there exists $S\in\mathcal{S}_f$ such that
  $e=\mathcal{X}_S$. In particular, we have that
  $\mathcal{B}(\overline{\Kset}_f)$ is a discrete subset of
  $\overline{\Kset}_f$.}}
\end{The}
\begin{proof} Let $\hat{e}=e(\varphi)$ be a representative any of
$e$. Then since $e^2=e$, it follows that $e(1-e)=0$, i.e.,
$\hat{e}(1-\hat{e})\in\mathcal{N}_f(\Kset)$. Hence there exists $p\in\Nset,
~\gamma\in\Gamma$ such that for all $q\ge p$ and for all
$\varphi\in\mathcal{A}_q(\Kset)$ there exists $C=C(\varphi)>0,
~\eta=\eta(\varphi)>0$ such that
\begin{equation}\label{null-1}
|\hat{e}(\varphi_\varepsilon)(1-\hat{e}(\varphi_\varepsilon))|\le
C\varepsilon^{\gamma(q)-p}, ~\forall~ 0<\varepsilon<\eta.
\end{equation}
Let $S=(N_a(\hat{e}))^c=\{\varphi\in\mathcal{A}_0(\Kset)||e(\varphi)|\ge
\dot{\alpha}_a(\varphi)\}$ and
$\hat{u}=\hat{e}-\hat{\mathcal{X}}_S$. Then we will go show that
$\hat{u}\in\mathcal{N}_f(\Kset)$. Indeed, if
$\hat{u}=\hat{e}-\hat{\mathcal{X}}_S$, then for all
$0<\varepsilon<\eta$, we have that 
$$
|\hat{u}(\varphi_\varepsilon)|=\left\{\begin{array}{ll}
|1-\hat{e}(\varphi_\varepsilon)|, & \mbox{se $\varphi_\varepsilon\in S$}\\
|\hat{e}(\varphi_\varepsilon)|, & \mbox{se $\varphi_\varepsilon\in S^c$}.
\end{array}\right.
$$
Indeed, if $\varphi_\varepsilon\in S$, then
$|\hat{e}(\varphi_\varepsilon)|\ge\dot{\alpha}_a(\varphi_\varepsilon)
=(i(\varphi))^a\varepsilon^a$, hence 
\begin{equation}\label{null-2}
\frac{1}{|\hat{e}(\varphi_\varepsilon)|}\le
  (i(\varphi))^{-a}\varepsilon^{-a}, ~\forall~ 0<\varepsilon<\eta,
\end{equation}
observing that
  $|\hat{e}(\varphi_\varepsilon)|\ne 0$. Therefore $e$ is a
  non-trivial idempotente. Thus, we have that
  \begin{eqnarray*}
    |\hat{u}(\varphi_\varepsilon)|&=&|1-\hat{e}(\varphi_\varepsilon)|\\
&=&\frac{1}{|\hat{e}(\varphi_\varepsilon)|}|\hat{e}(\varphi_\varepsilon)||1-\hat{e}(\varphi_\varepsilon)|\\
&\le&\frac{1}{|\hat{e}(\varphi_\varepsilon)|}|\hat{e}(\varphi_\varepsilon)(1-\hat{e}(\varphi_\varepsilon))|\qquad(\mbox{for
  (\ref{null-1}) and (\ref{null-2}), we have that})\\
&\le&((i(\varphi))^{-a}\varepsilon^{-a})(C\varepsilon^{\gamma(q)-p}),
~\forall ~0<\varepsilon<\eta\\
&=&C(i(\varphi))^{-a}\varepsilon^{\gamma(q)-(p+a)}, ~\forall ~0<\varepsilon<\eta.
\end{eqnarray*}
Hence, 
\begin{equation}\label{resul-1}
|\hat{u}(\varphi_\varepsilon)|\le
K\varepsilon^{\gamma(q)-P}, ~\forall~ 0<\varepsilon<\eta,
\end{equation}
where $K=C(i(\varphi))^{-a}> 0$ and $P=p+a$. Now, if $\varphi_\varepsilon\in
S^c=N_a(\hat{e}), ~\forall ~0<\varepsilon<\eta$, then
\begin{equation}\label{null-3}
|\hat{e}(\varphi_\varepsilon)|<\dot{\alpha}_a(\varphi_\varepsilon)=(i(\varphi))^a\varepsilon^a\Rightarrow
|\hat{e}(\varphi_\varepsilon)|<(i(\varphi))^a\varepsilon^a, ~\forall~ 0<\varepsilon<\eta. 
\end{equation}
However, in this case, we have that
$|\hat{u}(\varphi_\varepsilon)|=|\hat{e}(\varphi_\varepsilon)|$ and from
(\ref{null-3}), it follows that 
\begin{equation}\label{resul-2}
|\hat{u}(\varphi_\varepsilon)|\le(i(\varphi))^a\varepsilon^a, ~\forall~ 0<\varepsilon<\eta.
\end{equation}
Hence, from (\ref{resul-1}) and (\ref{resul-2}) we have that
$\hat{u}\in\mathcal{N}_f(\Kset)$ hence,
$\hat{e}=\hat{\mathcal{X}}_S$, i.e., $e=\mathcal{X}_S$. In
particular, we have that $\mathcal{B}(\overline{\Kset}_f)$ 
is a discrete subset of $\overline{\Kset}_f$. 
\end{proof}

\begin{Def}\label{von} {\tmsamp{A unitary ring is said to be Von
Neumann regular if all its principal ideals are generated by 
an idempotent.}}
\end{Def}

For our purposes the following is sufficient:
\begin{Pro}\label{nilo} {\tmsamp{Let $A$ be a commutative unitary 
ring and let $\mathcal{N}(A)$ be its nil radical. Then every prime 
ideal of $A$ is maximal iff $A/\mathcal{N}(A)$ is Von Neumann regular.}}
\end{Pro}
%\begin{proof}

%\end{proof}

\begin{Lem}\label{fator} {\tmsamp{Let $\gamma:\mathcal{A}_0(\Kset)\to
  \Rset\cup\{+\infty\}$ be defined as follows: 
\begin{displaymath}
\gamma(\varphi)=\left\{\begin{array}{ll}
+\infty, &\mbox{if} ~(i(\varphi))^{-1}\notin \Nset\\
p, &\mbox{if} ~(i(\varphi))^{-1}\in\Nset,
\end{array}\right.
\end{displaymath}
where $p$ is the smallest prime dividing $(i(\varphi))^{-1}$. Let
$x\in\overline{\Kset}_f$ such that
$\hat{x}(\varphi)=\hat{\dot{\alpha}}_{\gamma(\varphi)}(\varphi)$ is a 
representative. Then the principal ideal generated by $x$ is not an 
idempotent ideal.}}
\end{Lem}
\begin{proof} Suppose that $\mathfrak{J}=x\overline{\Kset}_f$ is an
ideal idempotente. Then from Theorem \ref{idemp} there exists 
$S\in\mathcal{S}_f$ such that
$\mathfrak{J}=\mathcal{X}_S\overline{\Kset}_f$. From the definition
of $\gamma$ we have two cases two to analyze:
\begin{enumerate}
\item[$(a)$] $\gamma(S)$ is finite;
\item[$(b)$] $\gamma(S)$ is infinite.
\end{enumerate}

$(a)$ Let $\gamma(S)$ be finite, $\sigma:=\max\{z|z\in\Rset\cap\gamma(S)\}$ 
and let us fix a prime number $p>\sigma$. Let also
$\varepsilon_n:=p^{-n}$. Then $\varepsilon_n\underset{n\rightarrow
  \infty}{\longrightarrow} 0$ and
$\gamma(\varphi_{\varepsilon_n})=p\notin\gamma(S), ~\forall~n\in\Nset$
hence, $\varphi_{\varepsilon_n}\notin S$, i.e.,
$\varphi_{\varepsilon_n}\in S^c$. Now, 
$$\hat{x}(\varphi_{\varepsilon_n})=\hat{\dot{\alpha}}_{\gamma(\varphi_{\varepsilon_n})}(\varphi_{\varepsilon_n})=\hat{\dot{\alpha}}_p(\varphi_{\varepsilon_n})=(i(\varphi))^p\varepsilon_n^p>0,
~\forall~ n\in\Nset$$ hence $x\mathcal{X}_{S^c}\ne
0$. Suppose that $x\in\mathfrak{J}=\mathcal{X}_S\overline{\Kset}_f$. Then
$x=y\mathcal{X}_{S}$ for some $y\in \overline{\Kset}_f$. 
From there, we have that 
$0=(y\mathcal{X}_S)\mathcal{X}_{S^c}=x\mathcal{X}_{S^c}\ne 0$ which 
is a contradiction.

$(b)$ Now, if $\gamma(S)$ is infinite, then there exists a sequence
$\{\varepsilon_n\}\subset\{\varepsilon|\varphi_\varepsilon\in S\}$
that converges to zero when $n\rightarrow \infty$ such that
$\{\gamma(\varphi_{\varepsilon_n})\}\subset \Nset$ is an increasing
and strict divergent sequence. Let us assume that $\mathcal{X}_S=yx$. Then
$$\hat{\mathcal{X}}_S(\varphi_\varepsilon)-\hat{y}(\varphi_\varepsilon)\hat{x}(\varphi_\varepsilon)=1-\hat{y}(\varphi_\varepsilon)\hat{x}(\varphi_\varepsilon)\underset{n\rightarrow
  \infty}{\longrightarrow} 0.$$ 
But since $\{\gamma(\varphi_{\varepsilon_n})\}$ is a divergent increasing 
sequence it follows easily that $y$ cannot be a moderate function, that is,
$y\notin\mathcal{E}_f^M(\Kset)$ which is a contradiction.

Of the two contradictions found in the cases $(a)$ and $(b)$ give the result.
\end{proof}

From the above we can affirm that $\overline{\Kset} _f$  is not Von
regular Neumann according to Definition \ref{von}. 
\begin{The}\label{nvon} {\tmsamp{$\overline{\Kset}_f$ is not Von 
Neumann regular. In particular, there exists
  $\mathcal{F}\in\mathit{P}_*(\mathcal{S}_f)$ such that
  $g_f(\mathcal{F})$ is not closed and $\overline{\Kset}_f$ has a
  prime ideal which is not maximal.}}
\end{The}
\begin{proof} We know from Theorem \ref{rad} that the nil-radical
  of $\overline{\Kset}_f$ is null, that is,
$\mathcal{N}(\overline{\Kset}_f)=\{0\}$. Therefore, from Proposition
\ref{nilo} and Lemma \ref{fator} follows that $\overline{\Kset}_f$ 
is not Von regular Neumann.
\end{proof}

\begin{The} {\tmsamp{For all
      $\mathcal{F}\in\mathit{P}_*(\mathcal{S}_f)$, we
have $g_f(\mathcal{F})$ is a prime ideal.}}
\end{The}
\begin{proof} Initially, let us assume that $\Kset=\Rset$. We need to
  show that condition $(ii)$ of Lemma \ref{hum} is 
  verified. Thus, there are 
$a,b\in\overline{\Rset}_f$ such that $a^-,b^-\in g_f(\mathcal{F})$. Then
\begin{eqnarray*}
  (ab)^-&=&\frac{1}{2}(ab-|a||b|)\\
&=&\frac{1}{2}[(a^++a^-)(b^++b^-)-(a^+-a^-)(b^+-b^-)]\\
&=&a^+b^-+a^-b^+.
\end{eqnarray*}
Hence $(ab)^-\in g_f(\mathcal{F})$. Therefore
$a^+b^-+a^-b^+\in g_f(\mathcal{F})$. This means that the positive cone
is invariant for multiplication. Now, let $a\in\overline{\Rset}_f$
such that $a^2\in g_f(\mathcal{F})$. From the definition of the ideal
$g_f(\mathcal{F})$, we have that there exists
$A\in\mathcal{F}$ such that $a^2=a^2\mathcal{X}_A$. Hence,
$a^2\mathcal{X}_{A^c}=0$ and thus
$a\mathcal{X}_{A^c}\in\mathcal{N}(\overline{\Rset}_f)=0$. Them,
$a=a\mathcal{X}_A+a\mathcal{X}_{A^c}=a\mathcal{X}_A\in
g_f(\mathcal{F})$. Thus, we are made in this case. To 
complete the test let us consider the case where $\Kset=\Cset$. 
Let $x,y\in g_f(\mathcal{F})$ such that $xy\in g_f(\mathcal{F})$. Then
$|xy|=|x||y|\in\overline{\Rset}_f\cap
g_f(\mathcal{F})={g_{f}}_r(\mathcal{F})$. For the first case $|x|$ or
$|y|$ belongs to ${g_{f}}_r(\mathcal{F})$ and from the convexity 
of ideals it follows the result.
\end{proof}

\begin{Cor} {\tmsamp{$\{{g_f}_r(\mathcal{F})|\mathcal{F}\in\mathit{P}_*(\mathcal{S}_f)\}$
  is the set of minimal prime ideals of $\overline{\Rset}_f$ and
  $\{g_f(\mathcal{F})|\mathcal{F}\in\mathit{P}_*(\mathcal{S}_f)\}$ is
  the set prime ideals of $\overline{\Cset}_f$.}}
\end{Cor}

{\bf{\underline{acknowledgement}}}
This paper is part of the second authors Ph.D. thesis done under
supervision of the last author at the University of São Paulo, Brazil.

\addcontentsline{toc}{chapter}{\bibname}
\bibliographystyle{mybibst}
\bibliography{biblio_arj_algebra}

\begin{thebibliography}{1}

\bibitem{afj}
{\sc J.~Aragona, R.~Fernandez, and S.~O. Juriaans}, {\em Natural topologies on
  {C}olombeau algebras}.
\newblock submited.

\bibitem{NTCA}
\leavevmode\vrule height 2pt depth -1.6pt width 23pt, {\em The sharp topology
  on the full {C}olombeau algebra of generalized functions}, Integral
  Transforms Spec. Funct., 17 (2006), pp.~165--170.

\bibitem{AGJ}
{\sc J.~Aragona, A.~R.~G. Garcia, and S.~O. Juriaans}, {\em Generalized
  solutions of a nonlinear parabolic equation with generalized functions as
  initial data}.
\newblock Submitted, 2008.

\bibitem{AJ}
{\sc J.~Aragona and S.~O. Juriaans}, {\em Some structural properties of the
  topological ring of {C}olombeau's generalized numbers}, Comm. Algebra, 29
  (2001), pp.~2201--2230.

\bibitem{AJOS}
{\sc J.~Aragona, S.~O. Juriaans, O.~R.~B. Oliveira, and D.~Scarpal\'ezos}, {\em
  Algebraic and geometric theory of the topological ring of {C}olombeau's
  generalized functions}.
\newblock Submitted.

\bibitem{MK}
{\sc M.~Kunziger}, {\em Lie transformation groups in Colombeau algebras}, PhD
  thesis, Fakult\"at der Universitat Wien, 1996.

\end{thebibliography}

\end{document}